\documentclass
{amsart}

\usepackage{color}
\usepackage{amsmath}
\usepackage{amssymb}
\usepackage{amsfonts}
\usepackage[abbrev]{amsrefs}
\usepackage{stmaryrd}
\usepackage{graphicx}
\usepackage{bm}
\usepackage[colorlinks=true,linkcolor=blue,urlcolor=blue]{hyperref}

\definecolor{red}{rgb}{1.0,0.0,0.0}

\definecolor{blu}{rgb}{0.0,0.0,1.0}

\definecolor{gre}{rgb}{0.03,0.50,0.03}

\newtheorem{theorem}{Theorem}[section]
\newtheorem{proposition}[theorem]{Proposition}
\newtheorem{lemma}[theorem]{Lemma}
\newtheorem{cor}[theorem]{Corollary}

\theoremstyle{definition}
\newtheorem{definition}[theorem]{Definition}

\theoremstyle{definition}
\newtheorem{remark}[theorem]{Remark}
\newtheorem{assumption}[theorem]{Assumption}

\numberwithin{equation}{section}

\newcommand{\bfone}{{\mathbf{1}}}
\newcommand{\eps}{{\varepsilon}}

\newcommand{\mail}[1]{\href{mailto:#1}{\normalfont\texttt{#1}}}

\newcommand{\abs}[1]{\left|{#1}\right|}
\newcommand{\norm}[1]{\lVert{#1}\rVert}

\newcommand{\N}{{\mathbb{N}}}

\newcommand{\R}{{\mathbb{R}}}

\newcommand{\dd}{{\mathrm{d}}}

\begin{document}
\title[Path-dependent Hamilton--Jacobi equations]
{Path-dependent Hamilton--Jacobi equations with
$u$-dependence and time-measurable Hamiltonians}

\thanks{The authors would like to thank the anonymous referees
for many valuable comments.}

\author[E.~Bandini]{Elena Bandini\textsuperscript{\MakeLowercase{a},1}}
\thanks{\noindent \textsuperscript{a} University of Bologna, Department of Mathematics, Piazza di Porta San Donato 5, 40126 Bologna (Italy).}
\author[C.~Keller ]{Christian Keller\textsuperscript{\MakeLowercase{b},2}}

\thanks{\noindent \textsuperscript{b} University of Central  Florida, Department of Mathematics,  4393 Andromeda Loop N
Orlando, FL 32816 (USA).
\\
\noindent \textsuperscript{1} E-mail: \mail{elena.bandini7@unibo.it}.
\\
\noindent \textsuperscript{2} E-mail: \mail{christian.keller@ucf.edu}.
}

\date{December 25, 2024}

\subjclass[2010]{}

\begin{abstract}
We establish  existence and uniqueness of minimax solutions  for a fairly
general class of path-dependent Hamilton--Jacobi equations.
In  particular, the relevant Hamiltonians can contain the solution and they
only need  to be measurable with respect to time. 
We apply our results
to optimal control problems  of (delay) functional differential equations with 
cost functionals that  have discount factors
and with  time-measurable data. 
Our main results are  also crucial for our companion paper  Bandini and Keller 
 [arXiv preprint arXiv:2408.02147 (2024)],
where  non-local path-dependent Hamilton--Jacobi--Bellman equations associated to the
 stochastic optimal control
of non-Markovian piecewise deterministic processes are studied.
\end{abstract}
\maketitle

\smallskip

\noindent \textbf{AMS 2020:}   35F21,  49L25, 34K35

\smallskip

\pagestyle{plain}

\section{Introduction}
We study path-dependent Hamilton--Jacobi equations of the form
\begin{align}\label{E:1st}
\partial_t u(t,x)+H(t,x,u(t,x),\partial_x u(t,x))=0,\quad (t,x)\in [0,T)\times D([0,T],\R^d),
\end{align}
with  time-measurable Hamiltonian, i.e., the mappings $t\mapsto H(t,x,y,z)$ only need  to be Borel-measurable.
The \emph{path space}  $D([0,T],\R^d)$ in \eqref{E:1st} is the set of all right-continuous
functions from $[0,T]$ to $\R^d$ that have left limits.
Important special cases of \eqref{E:1st} are Isaacs equations associated to differential games with history-dependent
data
and Hamilton--Jacobi--Bellman (HJB) 
equations associated to optimal controls problems involving (delay) functional differential
equations. {\color{black}
Consider, e.g.,
a value function $v$ of the form 
\begin{align*}
v(t_0,x_0)=\inf_{a(\cdot)\in\mathcal{A}} h(\{x^{t_0,x_0,a(\cdot)}(t)\}_{0\le t\le T}),\quad (t_0,x_0)\in [0,T]\times D([0,T],\R^d),
\end{align*}
where $x=x^{t_0,x_0,a(\cdot)}:[0,T]\to\R^d$ solves an equation of the form
\begin{align*}
x^\prime (t)&= f(t,x(t-1),a(t))\text{ on $(t_0,T)$, where $x(s):=x(0)$ in case $s<0$,}\\
x\vert_{[0,t_0]}&=x_0\vert_{[0,t_0]}.
\end{align*}
Then $v$ should formally solve 
the HJB equation
\begin{align*}
\partial_t v(t,x)+\inf_{a\in A} 
\left[ f(t,x({\max\{t-1,0\}}),a)\cdot \partial_x v(t,x)\right]&=0,\,\, (t,x)\in [0,T)\times D([0,T],\R^d),\\
v(T,x)&=h(x),\quad x\in D([0,T],\R^d).
\end{align*}}

{\color{black} Note that our path-dependent equations of the form  \eqref{E:1st} differ from usual partial differential equations 
in two crucial points.\footnote{I.e., \eqref{E:1st} is not a ``standard" partial differential equation
on an infinite-dimensional space.}
First, solutions $u$ of $\eqref{E:1st}$ as well as the Hamiltonian $H$ are required to be \emph{non-anticipating}, i.e.,
\begin{align}\label{E:Intro:NonAnt}
x\vert_{[0,t]}=\tilde{x}\vert_{[0,t]}\Longrightarrow u(t,x)=u(t,\tilde{x})\text{ and } H(t,x,y,z)=H(t,\tilde{x},y,z),
\end{align}
in other words $u(t,x)$ and $H(t,x,y,z)$ depend only on the history of the path $x$ until time $t$, i.e., on $\{x(s)\}_{0\le s\le t}$.
This requirement makes equations of the form \eqref{E:1st}  as well as their 2nd order counterparts
suitable for deterministic and stochastic optimal control problems with history dependent data
(the controller knows the history and the current state but not the future).
Second, the so-called \emph{path derivatives} $\partial_t$ and $\partial_x$ in \eqref{E:1st}
(originated in \cite{Kim85})
 are not understood in the usual (Fr\'echet) sense. Instead they are
  defined in a way such that a relatively simple chain rule holds, e.g.,
$\frac{d}{dt} u(t,x)=\partial_t u(t,x)+\partial_x u(t,x)\cdot x^\prime(t)$ 
for a non-anticipating ``smooth" function $u$ on $[0,T]\times
D([0,T],\R^d)$ and a differentiable path $x(\cdot)$ on $[0,T]$.
 Nevertheless, these path derivatives are compatible with the usual derivatives in the non-path-dependent case
(see below in section~\ref{S:PathDerivatives} for details concerning  path derivatives).}

{\color{black} 
Just as for ``standard" Hamilton--Jacobi equations on finite-dimensional spaces,
we do not expect to have smoothness, so we need non-smooth solutions.
The two most common notions of non-smooth solutions for path-dependent Hamilton--Jacobi equations
are viscosity solutions and minimax solutions.\footnote{ Minimax solutions originated in the 
theory of differential games (see, e.g., \cite{KrasovksiiSubbotin}),
where minmax operations are common, which motivates the name minimax solutions.
See also Remark~\ref{R:MinimaxSolutionMotivation}  below for more motivation.}
In fact, for a large class of path-dependent partial differential equations 
but with continuous (!) Hamiltonians, these two notions turn out to be equivalent (see \cite{GP23JFA}).\footnote{
Note that we do not pursue here a generalization of those equivalence results in our setting.
} 

Note that a large body of literature for non-smooth solutions of  {\color{black} path-dependent
partial differential equations} 
with continuous
Hamiltonians has been created (see \cite{GLP21AMO,GL24survey} and the references therein). 
{\color{black} Among the early works, we want to mention 
\cite{Lukoyanov2000,Lukoyanov03,Lukoyanov07,AubinHaddad02PPDE}
in the 1st order case
and \cite{EKTZ11,ETZ_I,ETZ_II}  in the 2nd order case.}
{\color{black} Also  note that t}ypically,  $C([0,T],\R^d)$ is used as path space. 
 But
 this 
 difference is not material.}

Here, we work with minimax solutions. This helps a lot in our setting,
 where $H$  in \eqref{E:1st} is time-measurable, as no change of {\color{black} definition} 
 is needed compared to the case of $H$
 being continuous. Furthermore, we  also extend the theory of minimax solutions 
for path-dependent equations in \cite{Lukoyanov03} to the case of $u$-dependent Hamiltonians and thereby provide
a counterpart to the theory of minimax solutions
for  
non-path-dependent
Hamilton--Jacobi equations with $u$-dependent Hamiltonians  in \cite{Subbotin}
(see also the references therein for earlier works but with more restrictive assumptions).

Viscosity solutions still play a useful role.  They are utilized for the proof of our comparison principle.\footnote{
We use the same methodology  for establishing existence and uniqueness for 
minimax solutions with the help of viscosity solution techniques
 as in \cite{BK18JFA}: The ``five-step-scheme" described on p.~2103 therein to be more precise.
}

Finally note that we choose $D([0,T],\R^d)$ as path space and  we work with relatively weak assumptions for $H$ 
 (in particular, time-measurability) because doing so is crucial for 
 establishing well-posedness of non-local path-dependent partial differential equations in our companion paper \cite{BK2}
(see section~1.2 therein for details).

\subsection{Further related literature}
Results for viscosity solutions for  
non-path-dependent equations with only time-measurable Hamiltonians were first published in \cite{Ishii85}
(this work had also  some influence for the proof of our comparison principle in section~\ref{S:Comparison}).
Further related early works are \cite{LionsPerthame87} and \cite{BarronJensen87}.

Minimax solutions for   
non-path-dependent
equations with time-measurable Hamiltonians were studied in 
\cite{SonEtAl95,SonEtAl97,minimaxBookTimeMeas, VanEtAl95}. 
Note that these works unlike ours require positive homogeneity of
the Hamiltonian with respect to the gradient.

In \cite{Qiu22SPA_RHJB}, a quite similar optimal control problem (compared to ours in section~\ref{S:OptimalControl})
is studied.
Besides being  time-measurable, the coefficients in \cite{Qiu22SPA_RHJB}  can even be random. Correspondingly, 
the controller minimizes an expected cost functional. However, in order to prove uniqueness for
the associated HJB equation, time-continuity is required in contrast to our work.
Moreover, we 
cover more general Hamilton--Jacobi equations besides HJB equations.

{\color{black}  Viscosity solutions for path-dependent Hamilton--Jacobi equations related
to  systems with ``distributed time delays"  only were investigated in \cite{Lukoyanov07} and
related to systems with  ``distributed time delays" as well as finitely many ``discrete time delays"
 in \cite{Lukoyanov10b}. 
A breakthrough in  \cite{zhou2020viscosity},  where for the first time  the smooth functional $\Psi$  
(see  \eqref{E:Comparison:Psi} below)
was studied,
led to much more general results, in particular, 
  existence and uniqueness for viscosity solutions of path-dependent HJB equations with Hamiltonians that only need to be 
  Lipschitz continuous in the path variable $x=x(\cdot)$ with respect to the supremum norm $x\mapsto \sup_{t\in [0,T]}
\abs{x(t)}$.  The mentioned functional $\Psi$ was crucial in \cite{zhou2020viscosity} to prove the comparison principle
for viscosity solutions and it is also crucial for our work (see Theorem~\ref{T:Comparison:H} and its proof below).}

\subsection{{\color{black} Some difficulties and their resolutions}}
To prove a comparison principle for viscosity solutions,  it is typical to 
consider the 
doubled equation 
\begin{align}\label{E:Naive}
\partial_t w+H(t,x,u,\partial_x w)-
H(t,\tilde{x},v,-\partial_{\tilde{x}} w)=0
\end{align}
 with $w(t,x,\tilde{x})=u(t,x)-v(t,\tilde{x})$, where $u$ is a viscosity sub- and $v$ is a viscosity supersolution
(cf.~\cite[section~4]{BK18JFA}\footnote{Note that there $u$ and $v$ are minimax semisolutions.}). However, the lack of continuity  of $H$ in $t$ 
causes trouble. Possible ways to deal with this issue are to replace 
$H(t,\cdot)$ in the definition of  test functions for viscosity subsolutions by
expressions  of the form $\varlimsup_{\delta\downarrow 0} \delta^{-1} \int_t^{t+\delta}
H(s,\cdot)\,ds$ (cf.~the treatment of the time-measurable operator $A(t,\cdot)$ in 
\cite[Definition~4.1]{BK18JFA}) or
of the form $\mathrm{ess}\varlimsup_{s\downarrow t} H(s,\cdot)$ (cf.~
\cite[Definition~4.2]{Qiu18SICON_SHJB}).\ In this work, 
we proceed with a different approach. Instead of the 
``naive" doubled equation~\eqref{E:Naive},
we consider another doubled equation\footnote{
See \eqref{E:Doubled:H:updated2} in Definition~\ref{D:wuv:viscosity:updated2}.
} with a relatively abstract Hamiltonian, which is at least semi-continuous
in time. Thereby, the difficulty due to the time-measurability of $H$ is circumvented.
 We do not know if 
  a proof of the comparison principle is possible if one proceeds along the previously mentioned other ``possible ways." 

{\color{black}
We use Perron's method to obtain our existence result. It should be noted that
continuity of the data is usually used in applying Perron's method 
for viscosity solutions (see \cite{Ishii87Perron})
and  we do not know how to overcome the lack of continuity
of the Hamiltonian with respect to time, especially
given that we work with   partial differential equations on an infinite-dimensional space.
Using minimax solutions circumvents this issue.}

\subsection{Organization of the rest of the paper}
Section~\ref{S:Setting} introduces the setting (path space with topology) and some notation.
In section~\ref{S:PathDerivatives}, we give meaning to the derivatives $\partial_t u$ and $\partial_x u$ in 
\eqref{E:1st}. Section~\ref{S:Minimax} contains the definition of minimax solution to terminal-value problems
involving \eqref{E:1st} as well as standing assumptions for our data such as the Hamiltonian $H$.
In section~\ref{S:Comparison}, a comparison principle for \eqref{E:1st} is established.
In section~\ref{S:Existence}, we establish a general existence result via Perron's method.
In section~\ref{S:OptimalControl}, we consider an optimal control problem for (delay) functional
differential equations with time-measurable data
and we show that the value function is the unique minimax solution to the associated
{\color{black} path-dependent} HJB equation.
{\color{black} In section~\ref{S:NonPPDE}, we consider non-path-dependent counterparts of our previous results.
In particular, we establish that the value function of an optimal control problem
with time-measurable data is the unique minimax solution of a standard non-path-dependent HJB equation.
Such a result is also new to the best of our knowledge.
Finally, in the appendix, we establish regularity of the value function for our 
optimal control problem in section~\ref{S:OptimalControl}.}

\section{Setting and notation}\label{S:Setting}
Let  $\Omega:= D([0,T],\R^d)$. 
We equip  $\Omega$ with the supremum norm $\norm{\cdot}_\infty$
and $[0,T]\times\Omega$ with a pseudo-metric $\mathbf{d}_\infty$ defined by
\begin{align*}
\mathbf{d}_\infty((t,x),(s,\tilde{x})):=\abs{t-s}+\sup_{0\le r\le T} \abs{x(r\wedge t)-\tilde{x}(r\wedge s)},
\end{align*}
{\color{black} where $r\wedge t:=\min\{r,t\}$ for any $r$, $t\in\R$.}

{\color{black}
\begin{remark}
Semicontinuous functions on $[0,T]\times\Omega$ equipped with $\mathbf{d}_\infty$  are 
automatically non-anticipating in the sense of \eqref{E:Intro:NonAnt}.
\end{remark}
}

Given a set $S\subset [0,T]$, we write $\bfone_S$ for the corresponding indicator function, i.e.,
$\bfone_S(t)=1$ if $t\in S$ and $\bfone_S(t)=0$ if $t\in [0,T]\setminus S$.

Given topological spaces $E$ and $F$, we denote by $C(E,F)$ the set of all continuous functions
from $E$ to $F$. In case $F=\R$, we just write $C(E)$. Similarly, we denote by $\mathrm{USC}(E)$
the set of all upper semicontinuous (u.s.c.) functions from $E$ to $\R$ and by
$\mathrm{LSC}(E)$ the set of all lower semicontinuous (l.s.c.) functions from $E$ to $\R$.


We write $a\cdot b=(a,b)$ for the inner product of two vectors $a$ and $b$ in $\R^d$.

Given $L\ge 0$, $(s_0,x_0)\in [0,T)\times\Omega$, define 
\begin{equation}\label{E:XLs0x0}
\begin{split}
&\mathcal{X}^L(s_0,x_0):=\Bigl\{x\in\Omega:\, \text{$x\vert_{[s_0,T]}$ is absolutely continuous
with}\\
&\qquad\qquad \abs{x^\prime(t)}\le L(1+\sup_{s\le t} \abs{x(s)})\text{ a.e.~on $(s_0,T)$
and $x=x_0$ on $[0,s_0]$}\Bigr\}.
\end{split}
\end{equation}
Those sets of  {\color{black} Lipschitz} 
paths are very important. In particular, thanks to their useful
properties listed in the following remarks, they  are helpful  insofar as they circumvent difficulties coming from the lack of local
compactness of $\Omega$.

\begin{remark}\label{R:X}
The sets $\mathcal{X}^L(s_0,x_0)$ are compact in $(\Omega,\norm{\cdot}_\infty)$
(see, e.g., Proposition~4.1 in \cite{Lukoyanov03} or Proposition~2.10 in \cite{BK18JFA}).
\end{remark}

\begin{remark}\label{R:Xn}
Let $L\ge 0$ and $(t_n,x_n)\to (t_0,x_0)$ in $[0,T]\times\Omega$ as $n\to\infty$.
Then every sequence $(\tilde{x}_n)_n$ with $\tilde{x}_n\in\mathcal{X}^L(t_n,x_n)$, $n\in\N$,
has a subsequence that converges to some $\tilde{x}_0\in\mathcal{X}^L(t_0,x_0)$
(cf.~Proposition~4.2 in \cite{Lukoyanov03} and Proposition~2.12 in \cite{BK18JFA}).
For a detailed proof, follow the approach of 
 Lemma~1 on page~87 in  \cite{FilippovBook}.
\end{remark}

\section{Path derivatives}\label{S:PathDerivatives}
Our path derivatives are  due to  A.~V.~Kim
(see \cite{KimBook} for a detailed exposition).

\begin{definition}\label{D:C111}
We write $\varphi\in C^{1,1,1}([0,T]\times\Omega\times\Omega)$ if
$\varphi\in C([0,T]\times\Omega\times\Omega)$ and there are functions
$\partial_t \varphi\in C([0,T]\times\Omega\times\Omega)$ and 
$\partial_x \varphi$, $\partial_{\tilde{x}}\varphi\in C([0,T]\times\Omega\times\Omega,\R^d)$,
called \emph{path derivatives} of $\varphi$, such
that, for every $(t_0,x_0,\tilde{x}_0)\in [0,T)\times\Omega\times\Omega$,
 every pair $(x,\tilde{x})\in
 {\color{black}\Omega\times\Omega}$
 with $(x,\tilde{x})\vert_{[0,t_0]}=(x_0,\tilde{x}_0)\vert_{[0,t_0]}$
 and  $(x,\tilde{x})\vert_{[t_0,T]}$ being Lipschitz continuous, and every $t\in (t_0,T]$, we have
 \begin{equation}\label{E:ChainRule}
 \begin{split}
& \varphi(t,x,\tilde{x})-\varphi(t_0,x_0,\tilde{x}_0)\\ &\qquad=\int_{t_0}^t 
\left[
\partial_t\varphi(s,x,\tilde{x}) 
+ x^\prime(s)\cdot\partial_x\varphi(s,x,\tilde{x})+\tilde{x}^\prime(s)\cdot\partial_{\tilde{x}} \varphi(s,x,\tilde{x})\right]\,\dd s.
 \end{split}
 \end{equation}
\end{definition}

\begin{remark}\label{R:pathDerivUnique}
The path derivatives of any function in  $C^{1,1,1}([0,T]\times\Omega\times\Omega)$ are uniquely determined
(see, e.g., Remark~2.17 in \cite{BK18JFA}).
\end{remark}

{\color{black} Independently, B.~Dupire introduced in \cite{dupirefunctional}
explicitly defined   path derivatives on $D([0,T],\R^d)$
and established a functional It\^o calculus.\footnote{I.e., a non-Markovian extension of the usual It\^o calculus.}
Dupire's derivatives 
$\partial^{\mathrm{Dupire}}_t$, $\partial^{\mathrm{Dupire}}_x
=(\partial^{\mathrm{Dupire}}_{x^1},\ldots, \partial^{\mathrm{Dupire}}_{x^d})$, and
$\partial^{\mathrm{Dupire}}_{\tilde{x}}
=(\partial^{\mathrm{Dupire}}_{\tilde{x}^1},\ldots, \partial^{\mathrm{Dupire}}_{\tilde{x}^d})$ for a
function $\varphi:[0,T]\times\Omega\times\Omega\to\R$ are
defined as follows (provided the  limits below exist and are finite):
\begin{equation}\label{E:DupireDerivatives}
\small
\begin{split}
\partial^{\mathrm{Dupire}}_t\varphi(t,x,\tilde{x})&:=
\lim_{\delta\downarrow 0}  \frac{\varphi(t+\delta,x(\cdot\wedge \delta),\tilde{x}(\cdot\wedge \delta))-\varphi(t,x,\tilde{x})}{\delta}
\text{ (provided $t<T$),}\\
\partial^{\mathrm{Dupire}}_t\varphi(T,x,\tilde{x})&:=\lim_{t\uparrow T} \partial^{\mathrm{Dupire}}_t \varphi(t,x,\tilde{x}),\\
\partial^{\mathrm{Dupire}}_{x^i}\varphi(t,x,\tilde{x})&:=
\lim_{r\to 0}  \frac{\varphi(t,(x^1,\ldots,x^{i-1},x^i+
r\bfone_{[t,T]},x^{i+1},\ldots,x^d),\tilde{x})-\varphi(t,x,\tilde{x})}{\delta},\\
\partial^{\mathrm{Dupire}}_{\tilde{x}^i}\varphi(t,x,\tilde{x})&:=
\lim_{r\to 0}  \frac{\varphi(t,x,(\tilde{x}^1,\ldots,\tilde{x}^{i-1},\tilde{x}^i+
r\bfone_{[t,T]},\tilde{x}^{i+1},\ldots,\tilde{x}^d))-\varphi(t,x,\tilde{x})}{\delta}
\end{split}
\end{equation}
 for each $i\in\{1,\ldots,d\}$. These derivatives
  coincide with the path derivatives in Definition~\ref{D:C111} in the following sense:
Suppose that 
$\varphi\in C([0,T]\times\Omega\times\Omega)$ is  continuously 
differentiable in the sense of Dupire,
i.e, $(\partial^{\mathrm{Dupire}}_t\varphi,\partial^{\mathrm{Dupire}}_x\varphi, 
 \partial^{\mathrm{Dupire}}_{\tilde{x}}\varphi)$ defined by \eqref{E:DupireDerivatives}
 exists and belong to $C([0,T]\times\Omega\times\Omega,\R\times\R^d\times\R^d)$. Then
 \cite[Theorem~2.4]{zhou2020viscosity} yields the chain rule \eqref{E:ChainRule}
 with $(\partial_t\varphi,\partial_x\varphi,\partial_{\tilde{x}}\varphi)$
 replaced by  $(\partial^{\mathrm{Dupire}}_t\varphi,\partial^{\mathrm{Dupire}}_x\varphi, 
 \partial^{\mathrm{Dupire}}_{\tilde{x}}\varphi)$. It follows immediately that
  $\varphi\in C^{1,1,1}([0,T]\times\Omega\times\Omega)$ and, together with Remark~\ref{R:pathDerivUnique},
  we can deduce that
 \begin{align*}
 \partial_t \varphi=\partial^{\mathrm{Dupire}}_t\varphi,\quad\partial_x \varphi=\partial^{\mathrm{Dupire}}_x\varphi,\quad
 \partial_{\tilde{x}}\varphi=\partial^{\mathrm{Dupire}}_{\tilde{x}}\varphi.
 \end{align*}
  }

\section{Minimax solutions}\label{S:Minimax}
Fix functions $H:[0,T]\times\Omega\times\R\times\R^d\to\R$ and $h:\Omega\to\R$. We consider the terminal-value problem
\begin{equation}\label{E:PPDE:H}
\begin{split}
-\partial_t u-H(t,x,u,\partial_x u)&=0,\quad (t,x)\in [0,T)\times\Omega,\\
 u(T,x)&=h(x),\quad x\in\Omega.
\end{split}
\end{equation}

The following two assumptions are always in force.

\begin{assumption}\label{A:h}
The function $h$ is continuous.
\end{assumption}

\begin{assumption}\label{A:H}
Suppose that $H$ satisfies the following conditions.

(i) For a.e.~$t\in (0,T)$, the function $(x,y,z)\mapsto H(t,x,y,z)$, $\Omega\times\R\times\R^d\to\R$,
is continuous.

(ii) For every $(x,y,z)\in\Omega\times\R\times\R^d$, the function $t\mapsto H(t,x,y,z)$, $[0,T]\to\R$,
is Borel measurable.

(iii) There is a constant $L_H\ge 0$ such that, for a.e.~$t\in (0,T)$, every $x\in\Omega$, $y\in\R$, $z$, $\tilde{z}\in\R^d$,
\begin{align*}
\abs{H(t,x,y,z)-H(t,x,y,\tilde{z})}\le L_H(1+\sup_{s\le t}\abs{x(s)})\abs{z-\tilde{z}}.
\end{align*}


(iv) For every $L\ge 0$, there exists a constant 
 $M_L\ge 0$ such that,
for every $(t_0,x_0)\in [0,T)\times\Omega$,   $x$, $\tilde{x}\in\mathcal{X}^L(t_0,x_0)$, $y\in\R$, $z\in\R^d$, and a.e.~$t\in (t_0,T)$,
\begin{align*}
\abs{H(t,x,y,z)-H(t,\tilde{x},y,z)}&\le M_L(1{\color{black}+\abs{y}}+\abs{z})\,\sup_{s\le t} \abs{x(s)-\tilde{x}(s)}.
\end{align*}

(v) For a.e.~$t\in (0,T)$ and every $(x,z)\in\Omega\times\R^d$, the function $y\mapsto H(t,x,y,z)$, $\R\to\R$, is non-increasing.

(vi) There is a constant $C_H\ge 0$ such that, for a.e.~$t\in (0,T)$  and all $(x,y)\in\Omega\times\R$,
\begin{align*}
\abs{H(t,x,y,0)}\le C_H(1+\sup_{s\le t}\abs{x(s)}+\abs{y}).
\end{align*}
\end{assumption}

Next, we introduce sets of paths needed in our definition of minimax solution.\footnote{
Note that without the $u$-dependence of $H$, the situation would be much easier. 
Only the sets  $\mathcal{X}^L(s_0,x_0)$ (and not $\mathcal{Y}^L
(s_0,x_0,y_0,z)$)
would then be needed (see, e.g., \cite[section~1.3]{BK18JFA}). 
}

Given $L\ge 0$,  $s_0\in [0,T)$, $x_0\in\Omega$, $y_0\in\R$, and $z\in\R^d$, define
\begin{align*}
\mathcal{Y}^L
(s_0,x_0,y_0,z)&:=\Bigl\{(x,y)\in\mathcal{X}^L(s_0,x_0)\times C([s_0,T]):\\
&\qquad\qquad y(t)=y_0+\int_{s_0}^t \left[{\color{black} x^\prime(s)\cdot z}-H(s,x,y(s),z)\right]\,\dd s\text{ on $[s_0,T]$}\Bigr\}.
\end{align*}

\begin{remark}\label{R:Y}
Thanks to Assumption~\ref{A:H}~(i) and (vi),
the sets  $\mathcal{Y}^L(s_0,x_0,y_0,z)$ are non-empty
(cf.~Proposition~2.13~(i) in \cite{BK18JFA}) and compact in $(\Omega\times C([s_0,T]),\norm{\cdot}_\infty)$
 (cf.~~Remark~\ref{R:X}). 
Moreover, using additionally Assumption~\ref{A:H}~(iii),
one can show that also
the intersections
$\mathcal{Y}^{L_H}(s_0,x_0,y_0,z)\cap
\mathcal{Y}^{L_H}(s_0,x_0,y_0,\tilde{z})$, $z\neq\tilde{z}$, are non-empty
(see, e.g., \cite[pages 73-74]{Subbotin} or  \cite[pages 2124-2125]{BK18JFA}).
\end{remark}

{\color{black} In the next remark, we use $D([0,T])$, the set of all right-continuous
functions from $[0,T]$ to $\R$ that have left limits.}

\begin{remark}\label{R:Yn}
Let $L\ge 0$, $z\in\R^d$, and $(t_n,x_n,y_n)\to (t_0,x_0,y_0)$ in $[0,T]\times\Omega\times\R$ as $n\to\infty$.
Then every sequence $(\tilde{x}_n,\tilde{y}_n)_n$ in $\Omega\times D([0,T])$
with $(\tilde{x}_n,\tilde{y}_n\vert_{[t_n,T]})\in\mathcal{Y}^L(t_n,x_n,y_n,z)$, $n\in\N$,
has a subsequence that converges to some $(\tilde{x}_0,\tilde{y}_0)$ in $(\Omega\times D([0,T]),\norm{\cdot}_\infty)$
with $(\tilde{x}_0,\tilde{y}_0\vert_{[t_0,T]})\in \mathcal{Y}^L(t_0,x_0,y_0,z)$.
This follows from Remark~\ref{R:Xn} and an appropriate
adaption of  the proof of Lemma~5 on page~8 in \cite{FilippovBook} to our setting.
\end{remark}

\begin{definition}\label{D:MinimaxSolution}
Let $L\ge 0$ and $u:[0,T]\times\Omega\to\R$.

(i) $u$ is a \emph{minimax $L$-supersolution} of \eqref{E:PPDE:H} if $u\in\mathrm{LSC}([0,T]\times\Omega)$,
if $u(T,\cdot)\ge h$ on $\Omega$, and if, for every $(s_0,x_0,z)\in [0,T)\times\Omega\times\R^d$, and
every $y_0\ge u(s_0,x_0)$, there exists an $(x,y)\in\mathcal{Y}^L(s_0,x_0,y_0,z)$ such that
$y(t)\ge u(t,x)$ for each $t\in [s_0,T]$.

(ii)  $u$ is a \emph{minimax $L$-subsolution} of \eqref{E:PPDE:H} if $u\in\mathrm{USC}([0,T]\times\Omega)$,
if $u(T,\cdot)\le h$ on $\Omega$, and if, for every $(s_0,x_0,z)\in [0,T)\times\Omega\times\R^d$, and
every $y_0\le u(s_0,x_0)$, there exists an $(x,y)\in\mathcal{Y}^L(s_0,x_0,y_0,z)$ such that
$y(t)\le u(t,x)$ for each  $t\in [s_0,T]$.

(iii) $u$ is a \emph{minimax $L$-solution} of \eqref{E:PPDE:H} if it is both a \emph{minimax $L$-supersolution}
and a \emph{minimax $L$-subsolution} of  \eqref{E:PPDE:H}.
\end{definition}

\begin{remark}\label{R:MinimaxSolutionMotivation}
{\color{black} Notice that classical solutions are minimax solutions. We sketch some of the arguments}
(adapted from  \cite[section~2.4 ]{Subbotin}).
Assume that $H$ is continuous and that $u$ is a \emph{classical solution} of \eqref{E:PPDE:H}, i.e.,
$u\in C^{1,1} ([0,T]\times\Omega)$ (this space is an obvious modification\footnote{
Just assume that $(t,x,\tilde{x})\mapsto u(t,x)\in C^{1,1,1}([0,T]\times\Omega)$.
} of Definition~\ref{D:C111}) and, for every $(t,x)\in [0,T)\times\Omega$,
\begin{align*}
-\partial_t u(t,x)-H(t,x,u(t,x),\partial_x u(t,x))= 0.
\end{align*}
We will show that $u$ is a minimax $L_H$-supersolution of \eqref{E:PPDE:H} 
with the slight modification that instead of $y_0\ge u(s_0,x_0)$
in Definition~\ref{D:MinimaxSolution}~(i), we only require to consider the case $y_0=u(s_0,x_0)$.
To this end, fix $(s_0,x_0,z)\in [0,T)\times\Omega\times\R^d$ and let $y_0= u(s_0,x_0)$.
Let $x$ be a solution of
\begin{align*}
x^\prime(t)=\begin{cases}
\frac{H(t,x,u(t,x),\partial_x u(t,x))-H(t,x,u(t,x),z)}{\abs{\partial_x u(t,x)-z}^2}
\cdot (\partial_x u(t,x)-z)&\text{if $\partial_x u(t,x)\neq z$,}\\
0 &\text{if $\partial_x u(t,x)=z$}
\end{cases} 
\end{align*}
a.e.~on $(s_0,T)$
with initial condition $x(t)=x_0(t)$ for every $t\in [0,s_0]$.
Note that $x\in\mathcal{X}^{(L_H)}(s_0,x_0)$ according to Assumption~\ref{A:H}.
Let $y$ be a solution of 
\begin{align*}
y^\prime(t)={\color{black} x^\prime(t)\cdot z}-H(t,x,u(t,x),z)\quad\text{a.e.~on $(s_0,T)$}
\end{align*}
with initial condition $y(s_0)=y_0$.
Since 
\begin{align*}
\frac{d}{dt} u(t,x)&=\frac{d}{dt} \int_{s_0}^t \partial_t u(s,x)+{\color{black} x^\prime(s)\cdot\partial_x u(s,x)}\,ds\\
&=-H(t,x,u(t,x),\partial_x u(t,x)) +
{\color{black} x^\prime(t)\cdot[\partial_x u(t,x)-z]}
 + {\color{black} x^\prime(t)\cdot z}\\
&=-H(t,x,u(t,x),z)+{\color{black} x^\prime(t)\cdot z}\\
&=y^\prime(t),
\end{align*}
we have
 $(x,y)\in\mathcal{Y}^{(L_H)}
 (s_0,x_0,y_0,z)$ and $y(t)=u(t,x)$ for all $t\in [s_0,T]$.
 Hence, $u$ is a minimax $L_H$-supersolution of \eqref{E:PPDE:H} 
 (in the slightly modified sense specified at the beginning of this remark).
\end{remark}

The next lemma provides an equivalent criterion for a function to be a minimax supersolution.
A corresponding statement holds for minimax subsolutions.

\begin{lemma}\label{R:Minimax:Quantifiers}
A function $u:[0,T]\times\Omega\to\R$ is a minimax $L$-supersolution of \eqref{E:PPDE:H} if and only if
$u\in\mathrm{LSC}([0,T]\times\Omega)$, $u(T,\cdot)\ge h$, and, for
each  $(s_0,x_0,z)\in [0,T)\times\Omega\times\R^d$, 
 $y_0\ge u(s_0,x_0)$, and $t\in (s_0,T]$, there  is an $(x,y)\in\mathcal{Y}^L(s_0,x_0,y_0,z)$ such that
$y(t)\ge u(t,x)$.
\end{lemma}
\begin{proof}
We only prove the non-trivial direction. We proceed along the lines of the proof of Lemma~3.6 in \cite{BK18JFA}.
Fix $(s_0,x_0,z)\in [0,T)\times\Omega\times\R^d$ and $y_0\ge u(s_0,x_0)$. Given is the following:
\begin{equation}\label{E:R:Minimax:Quantifiers}
\begin{split}
&\forall (s_1,x_1)\in [s_0,T)\times\Omega:\,\forall y_1\ge u(s_1,x_1):\,\forall t\in (s_0,T]:\\
&\qquad\exists (x,y)\in\mathcal{Y}^L(s_1,x_1,y_1,z):\, y(t)\ge u(t,x).
\end{split}
\end{equation}
We have to establish the existence of a pair $(x,y)\in\mathcal{Y}^L(s_0,x_0,y_0,z)$ independent from $t$
such that $y\ge u(\cdot,x)$ on $(s_0,T)$. To this end, consider a sequence $(\pi^m)_m$ of dyadic partitions
of $[s_0,T]$ with $\pi^m: s_0=t_0^m<t_1^m<\cdots<t^m_{n(m)}=T$, $m\in\N$, 
and $\sup_i (t^m_{i+1}-t^m_i)\to 0$ as $m\to\infty$. By \eqref{E:R:Minimax:Quantifiers}, there exists,
for each $m\in\N$, a finite sequence $(x^m_i,y^m_i)_{i=1}^{n(m)}$ such that we have
$(x^m_1,y^m_1)\in\mathcal{Y}^L(s_0,x_0,y_0,z)$ with $y^m_1(t^m_1)\ge u(t^m_1),x^m_1)$
and, for each $i\in\{1,\ldots,n(m)-1\}$, we have
$(x^m_{i+1}),y^m_{i+1})\in\mathcal{Y}^L(t^m_i,x^m_i,y^m_i(t^m_i),z)$ with 
$y^m_{i+1}(t^m_{i+1})\ge u(t^m_{i+1},x^m_{i+1})$. Using those sequences,
we define pairs $(x^m,y^m)\in\mathcal{Y}^L(s_0,x_0,y_0,z)$, $m\in\N$, by
\begin{align*}
x^m(t)=\begin{cases}
x_0&\text{if $0\le t\le s_0$,}\\
x^m_i(t)&\text{if $t^m_{i-1}< t\le t^m_i$},
\end{cases}
\quad\text{and}\quad
y^m(t)=\begin{cases}
y_0&\text{if $t= s_0$,}\\
y^m_i(t)&\text{if $t^m_{i-1}< t\le t^m_i$}.
\end{cases}
\end{align*}
By compactness\footnote{
Note that Assumption~\ref{A:H}~(vi) prevents a  possible blow up of $y^m$ as $m\to\infty$.
} of $\mathcal{Y}^L(s_0,x_0,y_0,z)$ (Remark~\ref{R:Y}), we can, without loss of generality, assume that
$(x^m,y^m)_m$ converges uniformly to a pair $(x^0,y^0)\in\mathcal{Y}^L(s_0,x_0,y_0,z)$.
Let $t\in (s_0,T]$. Then there is a sequence $(s^m)_m$ in $(s_0,T]$ with
$s^m\in\pi^m$, $m\in\N$,  that converges to $t$. Thus
$y^0(t)=\lim_m y^m(s^m)\ge\liminf_m u(s^m,x^m)\ge u(t,x^0)$ thanks to the lower semi-continuity of $u$.
This concludes the proof.
\end{proof}
\section{Comparison principle}\label{S:Comparison}
First, we introduce  viscosity solutions for a suitable doubled equation
(equation~\eqref{E:Doubled:H:updated2} below instead of the ``naive" doubled equation \eqref{E:Naive}).
Next, we establish connections between minimax solutions of   \eqref{E:PPDE:H} and those viscosity solutions.
Finally, we prove a comparison principle for our doubled equation, which immediately leads to  a comparison
principle between minimax sub- and supersolutions of \eqref{E:PPDE:H}.

We start by defining spaces of test functions, which are needed for our notion of viscosity solutions:
Given $L\ge 0$, $(s_0,x_0,\tilde{x}_0)\in [0,T)\times\Omega\times\Omega$, and a function
 $w:[0,T]\times\Omega\times\Omega\to\R$, let 
\begin{equation}\label{E:Testfunction:H}
\begin{split}
&\underline{\mathcal{A}}^L w(s_0,x_0,\tilde{x}_0):=
\Big \{\varphi\in C^{1,1,1}([s_0,T]\times\Omega\times\Omega):\,\exists \, T_0\in (t_0,T]:\\
&\qquad 0=(\varphi-w)(s_0,x_0,\tilde{x}_0)
 =\inf\{ (\varphi-w)(t,x,\tilde{x}): \\
 &\qquad\qquad\qquad  (t,x,\tilde{x})\in [s_0,T_0]\in\mathcal{X}^L(s_0,x_0)\times\mathcal{X}^L(s_0,\tilde{x}_0)\Big \}. 
\end{split}
\end{equation}

The next definition is vaguely inspired by \cite{Ishii85}.

\begin{definition}\label{D:wuv:viscosity:updated2}
Fix $L\ge 0$ and a function $\Upsilon:[0,T]\times\Omega\to\R$. Let $M_L$ be the constant from Assumption~\ref{A:H}~(iv). A function $w:[0,T]\times\Omega\times\Omega\to\R$  is 
 an \emph{$L$-viscosity subsolution of} 
\begin{equation}\label{E:Doubled:H:updated2}
\begin{split}
&\max\{w,0\}\cdot\Bigl[\partial_t w+M_L\left(1{\color{black}+\abs{\Upsilon}}+\abs{\partial_x w}\right)\,\sup\nolimits_{t\le s} \abs{x(t)-\tilde{x}(t)}\\
&\qquad\qquad\qquad\qquad {\color{black} + L_H\left(1+\sup\nolimits_{t\le s} \abs{\tilde{x}(t)}\right)\,
\abs{\partial_x w+\partial_{\tilde{x}} w}}\Bigr]
= 0
\end{split}
\end{equation}
\emph{on $[0,T)\times\Omega\times\Omega$ with parameter $\Upsilon$} if $w$ is u.s.c.,
{\color{black} $(s,x,\tilde{x})\mapsto \Upsilon(s,x)-w(s,x,\tilde{x})$,
$[0,T]\times\Omega\times\Omega\to\R$, is l.s.c.,}
and, 
 for 
  every
  $(s,x,\tilde{x})\in [0,T)\times\Omega\times\Omega$ 
   and
 every test function
 $\varphi\in\underline{\mathcal{A}}^L w(s,x,\tilde{x})$,  we have 
\begin{equation}\label{E:Viscosity:H:updated2}
\begin{split}
&\partial_t\varphi(s,x,\tilde{x})
+ M_L\left(1{\color{black}+\abs{\Upsilon(s,x)}}+
\abs{\partial_x\varphi(s,x,\tilde{x})}\right)
\,\sup_{t\le s} \abs{x(t)-\tilde{x}(t)}\\ 
&\qquad\qquad
{\color{black}
+ L_H\left(1+\sup\nolimits_{t\le s} \abs{\tilde{x}(t)}\right)\,\abs{\partial_x \varphi(s,x,\tilde{x}){\color{black}+}\partial_{\tilde{x}} \varphi(s,x,\tilde{x})}
}
\ge 0
\end{split}
\end{equation}
whenever $w(s,x,\tilde{x})> 0$.
\end{definition}

\begin{lemma}\label{L:Doubling:H}
Let $L\ge 0$,  
$u$ be a minimax $L$-subsolution,  and
$v$ be a minimax $L$-supersolution of \eqref{E:PPDE:H}. Then
$(t,x,\tilde{x})\mapsto w(t,x,\tilde{x}):=u(t,x)-v(t,\tilde{x})$,
$[0,T]\times\Omega\times\Omega\to\R$, is a viscosity $L$-subsolution
of \eqref{E:Doubled:H:updated2} {\color{black} with parameter $\Upsilon=u$}.
\end{lemma}
\begin{proof}
{\color{black} First note that $u=\Upsilon$ is u.s.c. and $v$ is l.s.c. Thus $w$ is u.s.c.~and
$(s,x,\tilde{x})\mapsto \Upsilon(s,x)-w(s,x,\tilde{x})=v(s,\tilde{x})$ is l.s.c.}

{\color{black} Next,} let $(s_0,x_0,\tilde{x}_0)\in [0,T)\times\Omega\times\Omega$, $w(s_0,x_0,\tilde{x}_0)>0$,
  $\varphi\in\underline{\mathcal{A}}^L w(s_0,x_0,\tilde{x}_0)$, and 
$T_0\in (s_0,T]$ such that \eqref{E:Testfunction:H} holds.
Let $y_0:= u(s_0,x_0)$, $\tilde{y}_0:= v(s_0,{\color{black}\tilde{x}_0})$, $z:=\partial_x\varphi(s_0,x_0,\tilde{x}_0)$, and $\tilde{z}:=-\partial_{\tilde{x}}\varphi(s_0,x_0,\tilde{x}_0)$.
By the minimax semisolution properties of $u$ and $v$, there exist $(x,y)\in\mathcal{Y}^L(s_0,x_0,y_0,z)$
and $(\tilde{x},\tilde{y})\in\mathcal{Y}^L(s_0,\tilde{x}_0,\tilde{y}_0,\tilde{z})$, such that, for every $t\in [s_0,T]$,
\begin{equation}\label{E:Doubling:H}
\begin{split}
&[u(t,x)-y_0]-[v(t,\tilde{x})-\tilde{y}_0]\\ &\qquad\ge [y(t)-y_0]-[\tilde{y}(t)-\tilde{y}_0]\\ &\qquad=
\int_{s_0}^t \left[(x^\prime(s),z)-H(s,x,y(s),z)-(\tilde{x}^\prime(s),\tilde{z})+
H(s,\tilde{x},\tilde{y}(s),\tilde{z})\right]\,\dd s. 
\end{split}
\end{equation}
By \eqref{E:Testfunction:H}, $(\varphi-w)(s_0,x_0,\tilde{x}_0)\le (\varphi-w)(t,x,\tilde{x})$ 
for every $t\in [s_0,T_0]$. Thus, the  chain rule applied to $\varphi$ together with \eqref{E:Doubling:H} yields
\begin{align*}
&\int_{s_0}^t \left[
\partial_t\varphi(s,x,\tilde{x})+(x^\prime(s),\partial_x\varphi(s,x,\tilde{x}))+
(\tilde{x}^\prime(s),\partial_{\tilde{x}}\varphi(s,x,\tilde{x}))
\right]\,\dd s\\
&\qquad\ge w(t,x,\tilde{x})-w(s_0,x_0,\tilde{x}_0)\\
&\qquad \ge \int_{s_0}^t \left[
(x^\prime(s),z)- (\tilde{x}^\prime(s),\tilde{z})-H(s,x,y(s),z)+
H(s,\tilde{x},\tilde{y}(s),\tilde{z})
\right]\,\dd s
\end{align*}
for every $t\in [s_0,T_0]$.  Next, let $\delta>0$ be sufficiently small such that $y(s)>\tilde{y}(s)$ for all $s\in [s_0,s_0+\delta]$. This is possible
because $w(s_0,x_0,\tilde{x}_0)>0$ and the functions $y$ and $\tilde{y}$ are continuous. Hence, by Assumption~\ref{A:H},
\begin{equation}\label{L:Doubling:H2}
\begin{split}
0&\le \int_{s_0}^{s_0+\delta} \Bigl[
\partial_t \varphi(s,x,\tilde{x})+H(s,x,y(s),z)-H(s,\tilde{x},\tilde{y}(s),\tilde{z})\\
&
\quad\quad + (x^\prime(s),\partial_x\varphi(s,x,\tilde{x})-z)+
(\tilde{x}^\prime(s),\partial_{\tilde{x}}\varphi(s,x,\tilde{x})+\tilde{z})
\Bigr]\,\dd s\\
&\le \int_{s_0}^{s_0+\delta}\Bigl\{ \partial_t \varphi(s,x,\tilde{x})+[H(s,x,y(s),z)-
H(s,\tilde{x},y(s),z)]\\ &\quad\quad +[H(s,\tilde{x},y(s),z)-H(s,\tilde{x},y(s),\tilde{z})]\\&\quad\quad+ [H(s,\tilde{x},y(s),\tilde{z})
-H(s,\tilde{x},\tilde{y}(s),\tilde{z})]\\ 
&\quad\quad +(x^\prime(s),\partial_x\varphi(s,x,\tilde{x})-z)+
(\tilde{x}^\prime(s),\partial_{\tilde{x}}\varphi(s,x,\tilde{x})+\tilde{z})
\Bigr\}\,\dd s\\
&\le \int_{s_0}^{s_0+\delta}\Bigl\{\partial_t \varphi(s,x,\tilde{x})+M_L(1+\abs{y(s)}+\abs{z})\sup_{t\le s} \abs{x(t)-\tilde{x}(t)}\\
&\quad\quad +L_H(1+\sup_{t\le s}\abs{\tilde{x}(t)})\abs{z-\tilde{z}} +0\\
&\quad\quad +(x^\prime(s),\partial_x\varphi(s,x,\tilde{x})-z)+
(\tilde{x}^\prime(s),\partial_{\tilde{x}}\varphi(s,x,\tilde{x})+\tilde{z})
\Bigr\}\,\dd s.
\end{split}
\end{equation}
Finally, dividing \eqref{L:Doubling:H2} by $\delta$ and letting $\delta\downarrow 0$ yields \eqref{E:Viscosity:H:updated2}.
\end{proof}

\begin{theorem}\label{T:Comparison:H}
Fix $L\ge 0$. Let  $w:[0,T]\times\Omega\times\Omega\to\R$ be a viscosity $L$-subsolution of \eqref{E:Doubled:H:updated2}
with an upper semi-continuous 
parameter $\Upsilon$. Suppose that $w(T,x,x)\le 0$ 
for every $x\in\Omega$.
Then $w(t,x,x)\le 0$ for every $(t,x)\in [0,T]\times\Omega$.
\end{theorem}

\begin{proof}
Assume that there is a point $(s_0,x_0)\in [0,T)\times\Omega$ such that
\begin{align*}
{\color{black}N_0}:=w(s_0,x_0,x_0)>0.
\end{align*}
Proceeding along the lines of the proof of Theorem~4.2 in \cite{BK18JFA}, we will obtain a contradiction.
The main difference compared to \cite{BK18JFA} is the choice of a different penalty functional.
Here, we use $\Psi:[s_0,T]\times\Omega\times\Omega\to\R$ defined by
\begin{align}\label{E:Comparison:Psi}
\Psi(s,x,\tilde{x}):=\begin{cases}
\frac{
{\color{black} (\sup\limits_{t\le s} \abs{x(t)-\tilde{x}(t)}^2-\abs{x(s)-\tilde{x}(s)}^2)^2}
}{
\sup\limits_{t\le s} \abs{x(t)-\tilde{x}(t)}^2
}+\abs{x(s)-\tilde{x}(s)}^2 &\text{if $\sup\limits_{t\le s} \abs{x(t)-\tilde{x}(t)}>0$,}\\
0&\text{otherwise.}
\end{cases}
\end{align}
This functional has been introduced in \cite{zhou2020viscosity}. {\color{black} By Lemma~2.8 in  \cite{zhou2020viscosity}, it follows that }
$\Psi\in C^{1,1,1}([s_0,T]\times\Omega\times\Omega)$ with derivatives
$\partial_t\Psi(s,x,\tilde{x})=0$ and
\begin{equation}\label{E:Comparison:Partial:x}
\begin{split}
&\partial_x \Psi(s,x,\tilde{x})=-\partial_{\tilde{x}} \Psi(s,x,\tilde{x})\\
&=\begin{cases}
{\tiny \left(2-\frac{
{\color{black} 4(\sup\limits_{t\le s} \abs{x(t)-\tilde{x}(t)}^2-\abs{x(s)-\tilde{x}(s)}^2)}
}{
\sup\limits_{t\le s} \abs{x(t)-\tilde{x}(t)}^2
}\right)\cdot
 [x(s)-\tilde{x}(s)]}&\text{if $\sup\limits_{t\le s} \abs{x(t)-\tilde{x}(t)}>0$,}\\
0&\text{otherwise.}
\end{cases}
\end{split}
\end{equation}

Given $\eps>0$, define $\Phi_\eps:[s_0,T]\times\Omega\times\Omega\to\R$ by
\begin{align}\label{E:Comp:Phi:eps}
\Phi_\eps(s,x,\tilde{x}):=w(s,x,\tilde{x})-\frac{T-s}{T-s_0}\cdot \frac{{\color{black}N_0}}{2}-\frac{1}{\eps} \Psi(s,x,\tilde{x}).
\end{align}
Fix a point $k_\eps=(s_\eps,x_\eps,\tilde{x}_\eps)$ at which the u.s.c.~map $\Phi_\eps$ attains a maximum on the
compact set 
$[s_0,T]\times\mathcal{X}^L(s_0,x_0)\times\mathcal{X}^L(s_0,x_0)$. Note that
\begin{align}\label{E:Comp:Meps}
{\color{black}N_\eps}:=\Phi_\eps(k_\eps)\ge \Phi_\eps(s_0,x_0,x_0)=w(s_0,x_0,x_0)-\frac{{\color{black}N_0}}{2}=\frac{{\color{black}N_0}}{2}.
\end{align}
Thus
\begin{align}\label{E1:DoubledCompar}
w(k_\eps)\ge \frac{{\color{black}N_0}}{2}+\frac{T-s_\eps}{T-s_0}\cdot\frac{{\color{black}N_0}}{2}+\frac{1}{\eps}\Psi(k_\eps)\ge \frac{{\color{black}N_0}}{2}>0.
\end{align}
Moreover, we have
\begin{align}\label{E:Prop3.7:UsersGuide}
{\color{black} \Psi(k_\eps)/\eps\to 0\quad\text{as $\eps\downarrow 0$}\quad\text{(cf.~Proposition~3.7 and its  
proof in \cite{CrandallIshiiLions})}}
\end{align}
and
{\color{black}thus}  $s_\eps<T$ 
if  $\eps$ is sufficiently small{\color{black}, which we shall assume from now on.}  

Next, define a map $\varphi_\eps:[s_\eps,T]\times\Omega\times\Omega\to\R$ by
\begin{align*}
\varphi_\eps(s,x,\tilde{x}):={\color{black}N_\eps}+\frac{T-s}{T-s_0}\cdot\frac{{\color{black}N_0}}{2}+\frac{1}{\eps} \Psi(s,x,\tilde{x}).
\end{align*}
Then $\varphi_\eps\in\underline{\mathcal{A}}^L w(k_\eps)$ with corresponding time $T_0=T$ because{\color{black},
by \eqref{E:Comp:Meps} and \eqref{E:Comp:Phi:eps},}
{\color{black}\begin{align*}
\varphi_\eps(k_\eps)-w(k_\eps) &={\color{black}N_\eps}+\frac{T-s_\eps}{T-s_0}\cdot\frac{{\color{black}N_0}}{2}+\frac{1}{\eps} \Psi(k_\eps)-w(k_\eps)\\
&=w(k_\eps)-\frac{T-s_\eps}{T-s_0}\cdot\frac{{\color{black}N_0}}{2}-\frac{1}{\eps}\Psi(k_\eps)
+\frac{T-s_\eps}{T-s_0}\cdot\frac{{\color{black}N_0}}{2}+\frac{1}{\eps} \Psi(k_\eps)-w(k_\eps)
\\ &= 0\\
&\le {\color{black}N_\eps}-\Phi_\eps(s,x,\tilde{x})
\\ & ={\color{black}N_\eps}+\frac{T-s}{T-s_0}\cdot\frac{{\color{black}N_0}}{2}+\frac{1}{\eps} \Psi(s,x,\tilde{x})-w(s,x,\tilde{x})
\\ &=\varphi_\eps(s,x,\tilde{x})-w(s,x,\tilde{x})
\end{align*}}
for every $(s,x,\tilde{x})\in [s_\eps,T]\times\mathcal{X}^L(s_\eps,x_\eps)\times
\mathcal{X}^L(s_\eps,\tilde{x}_\eps)$. Note that
\begin{align*}
\partial_t \varphi_\eps(s,x,\tilde{x})=-\frac{{\color{black}N_0}}{2(T-s_0)}\,\,\text{and}\,\,
\partial_x \varphi_\eps(s,x,\tilde{x})=-\partial_{\tilde{x}}\varphi_\eps(s,x,\tilde{x})=\frac{1}{\eps}\partial_x\Psi(s,x,\tilde{x}).
\end{align*}
Hence, since $w$ is a viscosity $L$-subsolution of \eqref{E:Doubled:H:updated2}
and  since \eqref{E1:DoubledCompar} holds, we have
\begin{equation}\label{E2:DoubledCompar}
\begin{split}
0&\le -\frac{{\color{black}N_0}}{2(T-s_0)}+M_L\left(1+\abs{\Upsilon(s_\eps,x_\eps)}+\frac{\abs{\partial_x\Psi(k_\eps)}}{\eps}\right)
\cdot \sup_{t\le s_\eps} \abs{x_\eps(t)-\tilde{x}_\eps(t)}\\
&\le  -\frac{{\color{black}N_0}}{2(T-s_0)}+\tilde{M}_L\left(1+\frac{\abs{\partial_x\Psi(k_\eps)}}{\eps}\right)
\cdot \sup_{t\le s_\eps} \abs{x_\eps(t)-\tilde{x}_\eps(t)}
\end{split}
\end{equation}
for 
a constant $\tilde{M}_L>0$ independent from $\eps$.
The second line of \eqref{E2:DoubledCompar} follows from
{\color{black}
\begin{align*}
\Upsilon(s_\eps,x_\eps)&\le \max\{\Upsilon(s,x):(s,x)\in [s_0,T]\times\mathcal{X}^L(s_0,x_0)\}<\infty\text{ and}\\
\Upsilon(s_\eps,x_\eps)&=w(k_\eps)+[\Upsilon(s_\eps,x_\eps)-w(k_\eps)] \\
&>0+ [\Upsilon(s_\eps,x_\eps)-w(k_\eps)]\qquad\qquad\qquad \text{(by \eqref{E1:DoubledCompar})}\\
& \ge \min\{\Upsilon(s,x)-w(s,x,\tilde{x}):(s,x,\tilde{x})\in [s_0,T]\times\mathcal{X}^L(s_0,x_0)\times
\mathcal{X}^L(s_0,x_0)\}\\ &>-\infty
\end{align*}
thanks to 
$\Upsilon$ being u.s.c., 
$(s,x,\tilde{x})\mapsto \Upsilon(s,x)-w(s,x,\tilde{x})$ being l.s.c., and $\mathcal{X}^L(s_0,x_0)$ being compact.}
It remains to note that, by \eqref{E:Comparison:Partial:x},  
\begin{align*}
\abs{\partial_x\Psi(s,x,\tilde{x})}\le 2\abs{x(s)-\tilde{x}(s)}
\end{align*}
 and also, by \cite[{\color{black} Lemma~2.8}]{zhou2020viscosity},
\begin{align*}
\frac{3-\sqrt{5}}{2}\,\sup_{t\le s} \abs{x(t)-\tilde{x}(t)}^2\le \Psi(s,x,\tilde{x})\le 2\, \sup_{t\le s} \abs{x(t)-\tilde{x}(t)}^2
\end{align*}
for every $(s,x,\tilde{x})\in [0,T]\times\Omega\times\Omega$. 
{\color{black} Together with \eqref{E:Prop3.7:UsersGuide}, we have }
\begin{align*}
\left(1+\frac{1}{\eps}\abs{\partial_x\Psi(k_\eps)}\right)
\cdot \sup_{t\le s_\eps} \abs{x_\eps(t)-\tilde{x}_\eps(t)}\le 
{\color{black} C}\sqrt{\Psi(k_\eps)}+\frac{C}{\eps}\Psi(k_\eps)\to 0\quad\text{as $\eps\downarrow 0$}
\end{align*}
for some constant $C>0$, which 
contradicts \eqref{E2:DoubledCompar}.
\end{proof}

Lemma~\ref{L:Doubling:H} and Theorem~\ref{T:Comparison:H}  immediately yield the following
result.
\begin{cor}\label{C:Comparison}
Let $L\ge 0$,  $u$ be a minimax $L$-subsolution, and
$v$ be a minimax $L$-supersolution of \eqref{E:PPDE:H}. Then $u\le v$ on $[0,T]\times\Omega$.
\end{cor}

\section{Existence}\label{S:Existence}
We show existence of minimax solutions via Perron's method (see \cite[section~V.2]{BardiCapuzzoDolcetta}
regarding a corresponding treatment for non-continuous viscosity solutions in the non-path-de\-pen\-dent case).
More precisely,  the scheme\footnote{
This scheme, which provides existence for minimax solutions, actually predates ``Perron's method" for viscosity
solutions introduced in \cite{Ishii87Perron}
(we refer to  the discussion in \cite[section~10.8]{Subbotin} for more details).
} in \cite[section~8]{Subbotin} is adapted 
to  path-dependent Hamilton-Jacobi equations.
In the case of Hamiltonians without $u$-dependence,
this has already been done in \cite[section~7]{Lukoyanov03} 
(cf.~also \cite[section~5]{BK18JFA},  
whose structure we 
 follow here).

\begin{definition} \label{D:NonContSol}
Let $L\ge 0$ and $u:[0,T]\times\Omega\to [-\infty,\infty]$.

(i) $u$ is a \emph{non-continuous minimax $L$-supersolution} 
of \eqref{E:PPDE:H} 
if $u(T,\cdot)= h$ on $\Omega$ and, for every $(t_0,x_0,z)\in [0,T)\times\Omega\times\R^d$, 
 $y_0> u(t_0,x_0)$, and $T_0\in (t_0,T]$,
 \begin{align}\label{E:NonContSuper}
 \inf_{(x,y)\in\mathcal{Y}^L(t_0,x_0,y_0,z)} \{u(T_0,x)-y(T_0)\}\le 0.
 \end{align}

ii) $u$ is a \emph{non-continuous minimax $L$-subsolution} 
of \eqref{E:PPDE:H} 
if $u(T,\cdot)= h$ on $\Omega$ and, for every $(t_0,x_0,z)\in [0,T)\times\Omega\times\R^d$, 
 $y_0< u(t_0,x_0)$, and $T_0\in (t_0,T]$,
 \begin{align}\label{E:NonContSub}
 \sup_{(x,y)\in\mathcal{Y}^L(t_0,x_0,y_0,z)} \{u(T_0,x)-y(T_0)\}\ge 0.
 \end{align}
 
 (iii) $u$ is a \emph{non-continuous minimax $L$-solution} 
of \eqref{E:PPDE:H} if it is both, a non-continuous minimax $L$-super-  and 
a non-continuous minimax $L$-subsolution of \eqref{E:PPDE:H}
\end{definition}

First, we establish existence of non-continuous minimax supersolutions. Following 
\cite[Proposition~8.6]{Subbotin},
 we define functions 
$\mu^z_+:[0,T]\times\Omega\times\R\to [-\infty,\infty]$ and
$u^z_+:[0,T]\times\Omega\to [-\infty,\infty]$, $z\in\R^d$, by
\begin{align*}
\mu^z_+(t_0,x_0,y_0)&:=\sup_{(x,y)\in\mathcal{Y}^{(L_H)}(t_0,x_0,y_0,z)}
\{h(x)-y(T)\},\\
u^z_+(t_0,x_0)&:=\sup\{r\in\R:\,\mu^z_+(t_0,x_0,r)\ge 0\}.
\end{align*}

\begin{lemma}\label{L:NonContSuper}
Let $z\in\R^d$. Then $u^z_+$ is a non-continuous minimax $L_H$-supersolution of 
\eqref{E:PPDE:H} and 
it is $[-\infty,\infty)$-valued.
\end{lemma}
\begin{proof}
(i) Boundary condition: Let $t_0=T$ and $x_0\in\Omega$. 
 Since $\mu^z_+(T,x_0,y_0)=h(x_0)-y_0$ for all $y_0\in\R$, 
we have $u^z_+(T,x_0)=\sup\{r\in\R:\, h(x_0)-r\ge 0\}=h(x_0)$.

(ii) Interior condition: Let $(t_0,x_0,\tilde{z})\in [0,T)\times\Omega\times\R^d$,
 $y_0>u^z_+(t_0,x_0)$, and $T_0\in (t_0,T]$.   
 Also pick a pair 
 $(\tilde{x},\tilde{y})\in\mathcal{Y}^{L_H}(t_0,x_0,y_0,z)
 \cap\mathcal{Y}^{L_H}(t_0,x_0,y_0,\tilde{z})$.
This is possible according to Remark~\ref{R:Y}.
To verify the interior minimax supersolution property,
it suffices to show that $u^z_+(T_0,\tilde{x})\le \tilde{y}(T_0)$.
 To this end, note first that $\mu^z_+(t_0,x_0,y_0)<0$ because otherwise
 $u^z_+(t_0,x_0)\ge y_0$, which would contradict $y_0> u^z_+(t_0,x_0)$.
 Therefore
 \begin{equation}\label{E3:L:NonContSuper}
 \begin{split}
 0&>\mu^z_+(t_0,x_0,y_0)\\
 &\ge \sup\{h(x)-y(T):\,(x,y)\in\mathcal{Y}^{(L_H)}(t_0,x_0,y_0,z)\\
 &\qquad\qquad\qquad\qquad\qquad\qquad
 \text{ and $(x,y)\vert_{[0,T_0]}=(\tilde{x},\tilde{y})\vert_{[0,T_0]}$}\}\\
 &=\sup\{h(x)-y(T):\, (x,y)\in\mathcal{Y}^{(L_H)}(T_0,\tilde{x},\tilde{y}(T_0),z)\}\\
 &=\mu^{z}_+(T_0,\tilde{x},\tilde{y}(T_0)).
\end{split}
 \end{equation}
Let us write $u^z_+(T_0,\tilde{x})=\sup R$, where
$R:=\{r\in\R:\,\mu^z_+(T_0,\tilde{x},r)\ge 0\}$. By \eqref{E3:L:NonContSuper}, $\tilde{y}(T_0)\not\in R$.
Note that, for every $s\ge 0$, we have
\begin{align}\label{E4:L:NonContSuper}
\mu^z_+(T_0,\tilde{x},\tilde{y}(T_0)+s)\le \mu^z_+(T_0,\tilde{x},\tilde{y}(T_0))-s\le 0
\end{align}
(this follows from Assumption~\ref{A:H}~(v) and can be shown exactly as equation~(8.4) in \cite{Subbotin}).
Thus $R\subseteq [-\infty,\tilde{y}(T_0))$, i.e., $u^z_+(T_0,\tilde{x})\le \tilde{y}(T_0)$.

(iii) $u^z_+$ is $[-\infty,\infty)$-valued: First, note that 
$\mu^z_+$ is $[-\infty,\infty)$-valued due to the compactness of the sets $\mathcal{Y}^L(t_0,x_0,y_0,z)$
(Remark~\ref{R:Y}) and the continuity of $h$ (Assumption~\ref{A:h}).
Now assume that $u^z_+(t_0,x_0)=\infty$ for some $(t_0,x_0)\in [0,T]\times\Omega$.
Then there is an increasing sequence $(r_n)_n$ in $\R$ with $\mu^z_+(t_0,x_0,r_n)\ge 0$ and
$r_n\to\infty$ as $n\to\infty$. But, by \eqref{E4:L:NonContSuper}, 
\begin{align*}
0\le \mu^z_+(t_0,x_0,r_n)=\mu^z_+(t_0,x_0,r_1+(r_n-r_1))\le \mu^z_+(t_0,x_0,r_1)-(r_n-r_1)\to -\infty
\end{align*}
as $n\to\infty$. This contradicts our assumption and thus concludes the proof.
\end{proof}

Define $u_0:[0,T]\times\Omega\to [-\infty,\infty]$ by
\begin{align*}
u_0(t,x):=\inf\{u(t,x):\,\text{$u$ is a non-continuous minimax $L_H$-supersolution
 of \eqref{E:PPDE:H}}\}.
\end{align*}

\begin{proposition}\label{P:NonContSol}
The function $u_0$ is a non-continuous minimax $L_H$-solution of \eqref{E:PPDE:H} and it is 
$\R$-valued.
\end{proposition}

\begin{theorem}\label{T:Perron}
The function $u_0$ is the unique minimax $L_H$-solution of \eqref{E:PPDE:H}.
\end{theorem}

\subsection{Proof of Proposition~\ref{P:NonContSol}}
The proof consists of four parts.

(i) $u_0$ is $\R$-valued: By the definition of $u_0$ and
Lemma~\ref{L:NonContSuper}, $u_0$ is $[-\infty,\infty)$-valued.
To show that $u_0$ is $(-\infty,\infty]$-valued, 
consider, following \cite[Proposition~8.4]{Subbotin} functions
$\mu^z_-:[0,T]\times\Omega\times\R\to [-\infty,\infty]$ and
$u^z:[0,T]\times\Omega\to [-\infty,\infty]$, $z\in\R^d$, defined by
\begin{align*}
\mu^z_-(t_0,x_0,y_0)&:= \inf_{(x,y)\in\mathcal{Y}^{(L_H)}(t_0,x_0,y_0,z)}
\{h(x)-y(T)\},\\
u^z_-(t_0,x_0)&:=\inf\{r\in\R:\,\mu^z_-(t_0,x_0,r)\le 0\}.
\end{align*}
Fix $z\in\R^d$. Note that one can show similarly as \eqref{E4:L:NonContSuper} that, for all $s\ge 0$, we have
\begin{align*}
\mu^z_-(t_0,x_0,y_0+s)\le \mu^z_-(t_0,x_0,y_0)-s,
\end{align*}
i.e., $u^z_-$ is $(-\infty,\infty]$-valued (cf.~part~(iii) of the proof of Lemma~\ref{L:NonContSuper}). 
Next, let $u$ be an arbitrary non-continuous minimax $L_H$-solution of \eqref{E:PPDE:H}. 
We show that, for every $(t,x)\in [0,T]\times\Omega$, we have $u(t,x)\ge u^z_-(t,x)$,
which then immediately yields $u_0(t,x)\ge u^z_-(t,x)>-\infty$.
To see this, assume that $u(t_0,x_0)<u^z_-(t_0,x_0)$ for some $(t_0,x_0)\in [0,T]\times\Omega$.
Then we can pick a $y_0\in (u(t_0,x_0),u^z_-(t_0,x_0))$ and thus, by \eqref{E:NonContSuper}
with $T_0=T$ and $u(T,\cdot)=h$, we have $\mu^z_-(t_0,x_0,y_0)\le 0$. But this implies
$u^z_-(t_0,x_0)\le y_0<u^z_-(t_0,x_0)$.
We can conclude that $u_0$ is $\R$-valued.

(ii) Boundary condition: By the definition of $u_0$ and Definition~\ref{D:NonContSol}, $u_0(T,\cdot)=h$.  

(iii) Interior minimax supersolution property: Let $(t_0,x_0,z)\in [0,T)\times\Omega\times\R^d$,
$y_0>u_0(t_0,x_0)$, and $T_0\in (t_0,T]$. By the definition of $u_0$ and
Lemma~\ref{L:NonContSuper}, there exists a non-continuous minimax $L_H$-supersolution $u$ of \eqref{E:PPDE:H}
such that $y_0>u(t_0,x_0)$. Thus, for every $n\in\N$, there is 
a pair $(x_n,y_n)\in\mathcal{Y}^{L_H}(t_0,x_0,y_0,z)$ such that $u(T_0,x_n)\le y_n(T_0)+n^{-1}$.
Hence $u_0(T_0,x_n)\le y_n(T_0)+n^{-1}$, i.e., $u_0$ is a non-continuous minimax $L_H$-supersolution 
of \eqref{E:PPDE:H}.

(iv) Interior minimax subsolution property: Let $(t_1,x_1,z)\in [0,T)\times\Omega\times\R^d$,
$y_1<u_0(t_1,x_1)$, and $T_1\in (t_1,T]$. We need to show that
\begin{align}\label{E:u0:NonContSub}
\sup_{(x,y)\in\mathcal{Y}^{L_H} (t_1,x_1,y_1,z)}\{u_0(T_1,x)-y(T_1)\}\ge 0.
\end{align}
To this end, we consider, following \cite[Proposition~8.5]{Subbotin}, the functions
$\mu^z:[0,T_1]\times\Omega\times\R\to [-\infty,\infty]$ and 
$u^z:[0,T]\times\Omega\to [-\infty,\infty]$ defined by
\begin{align*}
\mu^z(t_0,x_0,y_0)&:=\sup_{(x,y)\in\mathcal{Y}^{L_H}(t_0,x_0,y_0,z)}\{u_0(T_1,x)-y(T_1)\},\\
u^z(t_0,x_0)&:=\begin{cases}
\sup\{r\in\R:\,\mu^z(t_0,x_0,r)\ge 0\}\quad\text{if $t_0\le T_1$,}\\
u_0(t_0,x_0)\qquad\qquad\qquad\qquad\quad\,\,\,\,\text{if $t_0>T_1$.}
\end{cases}
\end{align*}
Assume momentarily that
\begin{align}\label{E:uz:NonContSuper}
\text{$u^z$ is a non-continuous minimax $L_H$-supersolution of \eqref{E:PPDE:H}.}
\end{align}
Then, by the definition of $u_0$, by part~(i) of this proof, and by noting that $t_0<T_1$,
we have $-\infty<u_0(t_1,x_1)\le u^z(t_1,x_1)=\sup\{r\in\R:\mu^z(t_1,x_1,r)\ge 0\}$.
Thus, recalling that $y_1<u_0(t_1,x_1)$, we can see that there exists an $r_1\in (y_0,\infty)$ such
that $\mu^z(t_1,x_1,r_1)\ge 0$ and then after showing similarly as \eqref{E4:L:NonContSuper}  that
\begin{align*}
 \mu^z(t_1,x_1,y_1+(r_1-y_1))\le \mu^z(t_1,x_1,y_1)-(r_1-y_0)
\end{align*}
 holds, we obtain \eqref{E:u0:NonContSub}.
Hence, it remains to establish \eqref{E:uz:NonContSuper},
which we do next.

(iv) (a) Boundary condition for $u^z$:

\emph{Case 1.} Let $T_1=T$.  Then  $\mu^z(T,x_0,y_0)=h(x_0)-y_0$
 and thus $u^z(T,x_0)=h(x_0)$.

\emph{Case 2.} Let $T_1<T$. Then $u^z(T,x_0)=u_0(T,x_0)=h(x_0)$.

(iv) (b) Interior condition for $u^z$: Fix $(t_0,x_0)\in [0,T)\times\Omega$, $y_0>u^z(t_0,x_0)$,
$\tilde{z}\in\R^d$, and $T_0\in (t_0,T]$. We need to show 
\begin{align}\label{E:uzNonContSuper:Interior}
\inf_{(x,y)\in\mathcal{Y}^{L_H}(t_0,x_0,y_0,\tilde{z})} \{u^z(T_0,x)-y(T_0)\}\le 0.
\end{align}

\emph{Case 1.} Let $T_1< t_0$.  Then $u_0=u^z$ on $[t_0,T]$ and \eqref{E:uzNonContSuper:Interior}
follows from $u_0$ being a non-continuous minimax $L_H$-supersolution of \eqref{E:PPDE:H}
(see part~(iii) of this proof).

\emph{Case 2.} Let $T_1=t_0$. 
Then $\mu^z(t_0,x_0,r)= u_0(t_0,x_0)-r$. Thus
$u_0=u^z$ on $[t_0,T]$  and \eqref{E:uzNonContSuper:Interior}
follows as in Case~1 of part (iv) (b) of this proof.

\emph{Case 3.} Let $t_0< T_0\le T_1$.
To obtain \eqref{E:uzNonContSuper:Interior},  follow the proof of
part (ii) of Lemma~\ref{L:NonContSuper}
but replace $h$ by $u_0(T_1,\cdot)$ and $T$ by $T_1$.

\emph{Case 4.} Let $t_0<T_1\le T_0$. Then $u^z(T_0,\cdot)=u_0(T_0,\cdot)$ and thus
\eqref{E:uzNonContSuper:Interior} follows by noting that
\begin{align*}
\inf_{(x,y)\in\mathcal{Y}^{L_H}(t_0,x_0,y_0,\tilde{z})} \{u_0(T_0,x)-y(T_0)\}
\le \inf_{(x,y)\in\mathcal{Y}^{L_H}(T_1,\tilde{x},\tilde{y}(T_1),\tilde{z})}
\{u_0(T_0,x)-y(T_0)\},
\end{align*}
where 
$(\tilde{x},\tilde{y})\in\mathcal{Y}^{L_H}(t_0,x_0,z)\cap\mathcal{Y}^{L_H}(t_0,x_0,\tilde{z})$
(cf.~Remark~\ref{R:Y}), 
and that $u_0$ is a 
 a non-continuous minimax $L_H$-supersolution of \eqref{E:PPDE:H}
(see part~(iii) of this proof)
together with $\tilde{y}(T_1)>u_0(T_1,\tilde{x})$.
Note that if the last inequality was not true, then
$\mu^z(t_0,x_0,y_0)\ge 0$ and thus $y_0\le u^z(t_0,x_0)$, which would be a contradiction.

This concludes the proof of Proposition~\ref{P:NonContSol}. \qed

\subsection{Proof of Theorem~\ref{T:Perron}}
Consider the u.s.c.~envelope $(u_0)^\ast:[0,T]\times\Omega\to [-\infty,\infty]$
and the l.s.c.~envelope $(u_0)_\ast:[0,T]\times\Omega\to [-\infty,\infty]$, which are defined by
\begin{align*}
(u_0)^\ast(t_0,x_0)&:=\inf_{\delta>0} \sup_{(t,x)\in O_\delta(t_0,x_0)} u_0(t,x),\\
(u_0)_\ast(t_0,x_0)&:=\sup_{\delta>0}\inf_{(t,x)\in O_\delta(t_0,x_0)} u_0(t,x),
\end{align*}
where $O_\delta(t_0,x_0)=\{(t,x)\in [0,T]\times\Omega:\,\mathbf{d}_\infty((t_0,x_0),(t_n,x_n))<\delta\}$.
We show that $(u_0)_\ast$ is a minimax $L_H$-supersolution of \eqref{E:PPDE:H}
and 
$(u_0)^\ast$ is a minimax $L_H$-subsolution of \eqref{E:PPDE:H}. 
The  comparison principle (Corollary~\ref{C:Comparison}) yields $(u_0)^\ast\le (u_0)_\ast$.
But, since, by definition, $(u_0)_\ast\le u_0\le (u_0)^\ast$ holds, we can conclude that 
$u_0$ is a minimax $L_H$-{\color{black}solution} 
 of \eqref{E:PPDE:H}. Again, by Corollary~\ref{C:Comparison},
it is the only one. 

Thus it remains to establish the minimax semisolution properties of the semicontinuous envelopes of $u_0$.

(i) $(u_0)^\ast$ is $(-\infty,\infty]$- and $(u_0)_\ast$ is $[-\infty,\infty)$-valued:
This follows from $u_0$ being $\R$-valued (Proposition~\ref{P:NonContSol}) and
$(u_0)_\ast\le u_0\le (u_0)^\ast$.

(ii) $(u_0)^\ast$ is $[-\infty,\infty)$- and $(u_0)_\ast$ is $(-\infty,\infty]$-valued:
We only show that $(u_0)^\ast<\infty$. The remaining part can be shown similarly.
Let $(t_0,x_0)\in [0,T]\times\Omega$ and $(t_n,x_n)_n$ be a sequence 
that converges to $(t_0,x_0)$ with respect to $\mathbf{d}_\infty$ such that
$(u_0)^\ast(t_0,x_0)=\lim_n u_0(t_n,x_n)$. Fix $z\in\R^d$.
By the definition of $u_0$ and Lemma~\ref{L:NonContSuper},
we have $\lim_n u_0(t_n,x_n)\le \lim_n u^z_+(t_n,x_n)=\lim_n r_n$
for some sequence $(r_n)_n$ in $\R$ with $\mu^z_+(t_n,x_n,r_n)\ge 0$.
By Remark~\ref{R:Y} and continuity of $h$, for each $n\in\N$, there is
a pair $(\tilde{x}_n,\tilde{y}_n)\in\mathcal{Y}^{L_H}(t_n,x_n,r_n,z)$
such that $\mu^z_+(t_n,x_n,r_n)=h(\tilde{x}_n)-\tilde{y}_n(T)$.
Without loss of generality, we can, by Remark~\ref{R:Xn}, assume that 
$(\tilde{x}_n)_n$ converges to some $\tilde{x}_0\in\mathcal{X}^{L_H}(t_0,x_0)$.
Assume now for the sake of a contradiction that $(u_0)^\ast(t_0,x_0)=\infty$.
Then we can assume, without loss of generality, that $(r_n)_n$ is
strictly increasing and strictly positive.
This in turn implies the following. Given $n\in\N$, fix a solution $y^0_n\in C([t_n,T])$
of $y^0_n(t)=\int_{t_n}^t [(\tilde{x}_n^\prime(s),z)-H(s,\tilde{x}_n,y^0_n(s),z)]\,\dd s$, 
$t\in [t_n,T]$. Then, we have $\tilde{y}_n> y^0_n$ on $(t_n,T)$,
as otherwise, by continuity of $y^0_n$ and $\tilde{y}_n$,
 there is a smallest time
$\tau_n\in (t_n,T]$ such that
$\tilde{y}_{n}(\tau)=y^0_n(\tau_n)$
and $\tilde{y}_n>y^0_n$  on $(t_n,\tau_n)$ but
then
\begin{align*}
0&=\tilde{y}_n(\tau)-y^0_n(\tau_n)=r_n+\int^\tau_{t_n} \left[
H(s,\tilde{x}_n,y^0_n(s),z)-H(s,\tilde{x}_{n},\tilde{y}_n(s),z)\right]\,\dd s\ge r_n>0
\end{align*}
(cf.~\cite[page~72]{Subbotin}), which is absurd. Using now $\tilde{y}_n> y^0_n$
together with $h(\tilde{x}_n)-\tilde{y}_n(T)\ge 0$, we obtain
\begin{equation}\label{E:rnii}
\begin{split}
r_n&\le h(\tilde{x}_n)+\int_{t_n}^T \left[H(s,\tilde{x}_n,\tilde{y}_n(s),z)-(\tilde{x}^\prime_n(s),z)\right]\,\dd s\\
&\le h(\tilde{x}_n)+\int_{t_n}^T \left[H(s,\tilde{x}_n,y^0_n(s),z)-(\tilde{x}^\prime_n(s),z)\right]\,\dd s.
\end{split}
\end{equation}
Note that, since $(t_n,x_n,y^0_n(t_n))=(t_n,x_n,0)\to (t_0,x_0,0)$ as $n\to\infty$, we
can, thanks to Remark~\ref{R:Yn}, assume that $(\tilde{x}_n,\bfone_{[t_n,T]}\,y_n^0)\to
(\tilde{x}_0,y_0^0)$ as $n\to\infty$ for some $y_0^0\in D([0,T])$ with
$(\tilde{x}_0,y_0^0\vert_{[t_0,T]})\in\mathcal{Y}^L(t_0,x_0,0,z)$. Thus, letting $n\to\infty$
in \eqref{E:rnii} yields
\begin{align*}
\infty\le h(\tilde{x}_0)+ \int_{t_0}^T \left[H(s,\tilde{x}_0,y^0_0(s),z)-(\tilde{x}^\prime_0(s),z)\right]\,\dd s<\infty,
\end{align*}
which is absurd.\footnote{
Alternatively, one can estimate $H$ by using $z=0$ and Assumption~\ref{A:H}~(vi) to obtain $\infty\le C<\infty$
for some constant $C$.
} Therefore, $(u_0)^\ast<\infty$.

(iii) Boundary conditions:
We only show $(u_0)^\ast(T,\cdot)\le h$. Proving $(u_0)_\ast\ge h$ can be done similarly. 
Let $x_0\in\Omega$ and  $(t_n,x_n)_n$ be a sequence in $[0,T]\times\Omega$ that converges to 
$(T,x_0)$
and that satisfies $\lim_n u_0(t_n,x_n)=(u_0)^\ast (T,x_0)$.
For each $n\in\N$, there is, by Proposition~\ref{P:NonContSol}, 
 an $(\tilde{x}_n,\tilde{y}_n)\in\mathcal{Y}^{L_H}(t_n,x_n,u_0(t_n,x_n)-1/n,z)$ such
 that $\tilde{y}_n(T)\le u_0(T,\tilde{x}_n)=h(\tilde{x}_n)$. By Remark~\ref{R:Yn},
 we can assume that the sequence $(\tilde{x}_n,t\mapsto \bfone_{[t_n,T]}(t)\,\tilde{y}_n(t))_n$
 converges in $(\Omega\times D([0,T]),\norm{\cdot}_\infty)$ to a pair $(\tilde{x},\tilde{y})$
 with $(\tilde{x},\tilde{y}\vert_{\{T\}})\in\mathcal{Y}^{L_H}(T,x_0,(u_0)^\ast(T,x_0),z)$,
 i.e., $\tilde{y}(T)=(u_0)^\ast(T,x_0)$ and $\tilde{x}=x_0$. Thus
 $(u_0)^\ast(T,x_0)=\lim_n \tilde{y}_n(T)\le \lim_n h(\tilde{x}_n)=h(x_0)$ thanks to the
   continuity of $h$ (Assumption~\ref{A:h}).

(iv) Interior conditions: We only establish the  minimax $L_H$-subsolution  property of
$(u_0)^\ast$. Our argument is very close to  \cite[page~78]{Subbotin}.
Let $(t_0,x_0,z)\in [0,T)\times\Omega\times\R^d$,
$y_0\le (u_0)^\ast (t_0,x_0)$, and $T_0\in (t_0,T]$.
Let $(t_n,x_n)_n$ be a sequence in $[0,T)\times\Omega$ that converges to $(t_0,x_0)$
and satisfies  $\lim_n u_0(t_n,x_n)=(u_0)^\ast(t_0,x_0)$  and $t_n<T_0$  for every $n\in\N$.
Next, put $y_n:=u_0(t_n,x_n)-1/n+y_0-(u_0)^\ast(t_0,x_0)$, $n\in\N$, so that $y_n<u_0(t_n,x_n)$
and
$(t_n,x_n,y_n)\to (t_0,x_0,y_0)$ as $n\to\infty$. 
 By Proposition~\ref{P:NonContSol},
$u_0$ is a non-continuous minimax $L_H$-subsolution of \eqref{E:PPDE:H} and thus
we obtain from  $y_n<u_0(t_n,x_n)$, $n\in\N$, the existence of a pairs
$(\tilde{x}_n,\tilde{y}_n)\in\mathcal{Y}^L(t_n,x_n,y_n,z)$ that satisfy
$\tilde{y}_n(T_0)-1/n\le u_0(T_0,\tilde{x}_n)$.
By Remark~\ref{R:Yn} and $\lim_n (t_n,x_n,y_n)=(t_0,x_0,y_0)$,
we can assume that $(\tilde{x}_n, t\mapsto \bfone_{[t_n,T]}(t)\,\tilde{y}_n(t))_n$
converges in $(\Omega\times D([0,T]),\norm{\cdot}_\infty)$ to
some pair $(\tilde{x},\tilde{y})$ with $(\tilde{x},\tilde{y}\vert_{[t_0,T]})\in\mathcal{Y}^{L_H}(t_0,x_0,y_0,z)$. Thus
\begin{align*}
\tilde{y}(T_0)=\lim_n \tilde{y}_n(T_0)\le\limsup_n u_0(T_0,\tilde{x}_n)\le (u_0)^\ast (T_0,\tilde{x}).
\end{align*}
By Lemma~\ref{R:Minimax:Quantifiers} (or, more precisely, its counterpart
for minimax subsolutions), the minimax $L_H$-subsolution  property of
$(u_0)^\ast$ has been established.

This concludes the proof of Theorem~\ref{T:Perron}. \qed

\section{Optimal control with discount factors}\label{S:OptimalControl}
Let $A$ be a {\color{black}non-empty subset of a Polish space.} 
Suppose that $A$ is the countable union of compact metrizable subsets of $A$
(cf.~Chapter~21 of \cite{CohenElliott}). Let $\mathcal{A}$ be the set of all Borel measurable functions
from $[0,T]$ to $A$. 
Further data for our optimal control problem are
functions $f:[0,T]\times\Omega\times A\to\R^d$, $\lambda:[0,T]\times\Omega\times A\to \R_+$,
$\ell:[0,T]\times\Omega\times A\to\R$, and $h:\Omega\to\R$.

{\color{black} The following assumption is in force in this section.}

\begin{assumption}\label{A:DOC}
Suppose that $f$, $\lambda$, $\ell$, and $h$ satisfy the following conditions.


(i) For a.e.~$t\in [0,T]$, 
{\color{black} the map $(x,a)\mapsto (f,\lambda,\ell)(t,x,a)$, $\Omega\times A\to\R^d\times\R_+\times\R$,}
is continuous. 

(ii) For every $(x,a)\in\Omega\times A$, the map $t\mapsto (f,\lambda,\ell)(t,x,a)$, $[0,T]\to\R^d\times\R_+\times\R$,
is Borel measurable.

(iii) There are constants $C_f$, $C_\lambda\ge 0$ such that, for a.e.~$t\in (0,T)$ and each $x\in\Omega$, 
\begin{align*}
\sup_{a\in A}\left(\abs{f(t,x,a))}+\abs{\ell(t,x,a)}\right)\le C_f(1+\sup_{s\le t}\abs{x(s)})
\text{ and } \sup_{a\in A} \abs{\lambda(t,x,a)}\le C_\lambda.
\end{align*}

(iv) There is a constant $L_f\ge 0$ such that, for a.e.~$t\in (0,T)$ and every $x$, $\tilde{x}\in\Omega$, 
\begin{align*}
&\sup_{a\in A}\left(\abs{f(t,x,a)-f(t,\tilde{x},a)}+
\abs{\ell(t,x,a)-\ell(t,\tilde{x},a)}+\abs{\lambda(t,x,a)-\lambda(t,\tilde{x},a)}\right)\\ 
&\qquad \le L_f\,\sup_{s\le t}\abs{x(s)-\tilde{x}(s)}{\color{black}\text{ and}}\\
&  {\color{black} \abs{h(x)-h(\tilde{x})}\le L_f\,\sup_{s\le T}\abs{x(s)-\tilde{x}(s)}}.
\end{align*}
\end{assumption}

For each $(t_0,x_0,\alpha)\in [0,T)\times\Omega\times\mathcal{A}$, 
let $\phi^{t_0,x_0,\alpha}=\phi$ be the solution of
\begin{align}\label{E:phi:prime:f}
\phi^\prime(t)&=f(t,\phi,\alpha(t))\text{ a.e.~on $(t_0,T)$},\quad
\phi(t)=x_0(t)\text{ on $[0,t_0]$}.
\end{align} 
Moreover, define, for each $(s,x,a)\in [0,T)\times\Omega\times\mathcal{A}$ and $t\in [s,T]$, 
\begin{align*}
\lambda^{s,x,a}(t)&:=\lambda(t,\phi^{s,x,a},a(t)),\\
\chi^{s,x,a}(t)&:=\exp\left(
-\int_s^t \lambda^{s,x,a}(r)\,\dd r
\right),\\
\ell^{s,x,a}(t)&:=\ell(t,\phi^{s,x,a},a(t)).
\end{align*}

Our goal is to 
show that the value function $v:[0,T]\times\Omega\to\R$ defined by
\begin{align}\label{E:value}
v(s,x):=\inf_{a\in\mathcal{A}} \left[\int_s^T \chi^{s,x,a}(t)\,\ell^{s,x,a}(t)\,\dd t+ \chi^{s,x,a}(T)\,h(\phi^{s,x,a})\right]
\end{align}
is a minimax solution of
\begin{equation}\label{E:DiscountedHJB}
\begin{split}
-\partial_t u(t,x)-H(t,x,u(t,x),\partial_x u(t,x))&=0,\quad (t,x)\in [0,T)\times\Omega,\\
u(T,x)&=h(x),\quad x\in\Omega,
\end{split}
\end{equation}
where $H:[0,T]\times\Omega\times\R\times\R^d\to\R$ is defined by 
\begin{align}\label{E:DOC:H}
H(t,x,y,z):=
{\color{black} \inf_{a\in A}}
 \left[\ell(t,x,a)+(f(t,x,a),z)-\lambda(t,x,a)\,y\right].
\end{align}
\begin{remark}
Suppose that Assumption~\ref{A:DOC} holds. 
One can show that
$H$ defined by \eqref{E:DOC:H} satisfies Assumption~\ref{A:H} with $L_H=C_f$ and
$M_L$ independent from $L$.
{\color{black} Note that time-measurability of $H$ follows from 
\cite[Theorem~8.2.11]{AubinFrankowska_setValued}.}
\end{remark}

The 
 dynamic programming principle holds. Its proof is standard and is omitted.
\begin{proposition}\label{P:DiscountedDPP}
If $0\le s\le t\le T$ and $x\in\Omega$, then
\begin{align*}
v(s,x)=\inf_{a\in\mathcal{A}} \left[\int_s^t \chi^{s,x,a}(r)\,\ell^{s,x,a}(r)\,\dd r +\chi^{s,x,a}(t)\,v(t,\phi^{s,x,a})\right].
\end{align*}
\end{proposition}

{\color{black}The proof of the next result can be found 
in appendix~\ref{S:Appendix:Regularity}.}

\begin{proposition}\label{P:value:regular}
For each $L\ge 0$ and $x_0\in\Omega$, there is a $C_{L,x_0}\ge 0$,
such that  $0\le t_0\le t_1\le T$ and $x$, $\tilde{x}\in\mathcal{X}^L(t_0,x_0)$ imply
$\abs{v(t_0,x)-v(t_1,x)}\le C_{L,x_0} (t_1-t_0)$ and 
$\abs{v({\color{black} t_1},x)-v({\color{black} t_1},\tilde{x})}\le C_{L,x_0} \sup_{s\le {\color{black} t_1}}\abs{x(s)-\tilde{x}(s)}$.
Moreover, $v$ is continuous.
\end{proposition}


\begin{lemma}\label{L:value:supersol}
Let $u$ be a minimax $L_f$-subsolution.
Let $v$ be the value function defined by \eqref{E:value}.
Then the function $(t,x,\tilde{x})\mapsto w(t,x,\tilde{x}):=u(t,x)-v(t,\tilde{x})$,
$[0,T]\times\Omega\times\Omega\to\R$, is a viscosity $C_f$-subsolution of 
\eqref{E:Doubled:H:updated2}  with parameter
$\Upsilon=u$.
\end{lemma}

\begin{proof}
{\color{black} Since  $u$ is u.s.c.~and $v$ is continuous (Proposition~\ref{P:value:regular}),
 the function $w$ is u.s.c.~and the function
$(s,x,\tilde{x})\mapsto \Upsilon(s,x)-w(s,x,\tilde{x})=v(s,\tilde{x})$ is l.s.c.,
i.e., we have the regularity  required by Definition~\ref{D:wuv:viscosity:updated2}.}

{\color{black}Now,} fix $(s_0,x_0,\tilde{x}_0)\in [0,T)\times\Omega\times\Omega$ and
$\varphi\in\underline{\mathcal{A}}^{(C_f)} w(s_0,x_0,\tilde{x}_0)$
with corresponding time $T_0\in (s_0,T]$. Let $w(s_0,x_0,\tilde{x}_0)>0$. Put
\begin{align*}
y_0:=u(s_0,x_0),\, \tilde{y}_0:=v(s_0,x_0),\,
z:=\partial_x \varphi(s_0,x_0,\tilde{x}_0),\, \tilde{z}:=-\partial_{\tilde{x}} \varphi(s_0,x_0,\tilde{x}_0).
\end{align*}
Also, for every $a\in\mathcal{A}$, put
\begin{align*}
\phi^a:=\phi^{s_0,\tilde{x}_0,a},\,\lambda^a:=\lambda^{s_0,\tilde{x}_0,a},\,
\chi^a:=\chi^{s_0,\tilde{x}_0,a},\,\ell^a:=\ell^{s_0,\tilde{x}_0,a},\,
f^a:=f(\cdot,\phi^a,a(\cdot)).
\end{align*}
First note that, due to the minimax $C_f$-subsolution property of $u$, we
can fix a pair $(x,y)\in\mathcal{Y}^{C_f}(s_0,x_0,y_0,z)$ such that, for all $t\in [s_0,T]$, we have
\begin{align}\label{E:value:supersol1}
u(t,x)-y_0\ge y(t)-y_0=\int_{s_0}^t \left[(x^\prime(s),z)-H(s,x,y(s),z)\right]\,\dd s.
\end{align}
Next, let $\delta\in (0,T_0-s_0]$ and let $\eps>0$. By Proposition~\ref{P:DiscountedDPP},
there exists a control $a=a^{\delta,\eps}\in\mathcal{A}$ such that
\begin{align*}
\tilde{y}_0=v(s_0,x_0)>\chi^a(s_0+\delta)\,v(s_0+\delta,\phi^a)+\int_{s_0}^{s_0+\delta}\chi^a(s)\,\ell^a(s)\,\dd s
 -\delta\,\eps.
\end{align*}
Thus
\begin{equation}\label{E:value:supersol2}
\begin{split}
&-\left[v(s_0+\delta,\phi^a)-\tilde{y}_0\right]\\ &\qquad>
-\delta\,\eps+[\chi^a(s_0+\delta)-1]\,v(s_0+\delta,\phi^a)+\int_{s_0}^{s_0+\delta}\chi^a(s)\,\ell^a(s)\,\dd s\\
&\qquad  =
-\delta\,\eps+\int_{s_0}^{s_0+\delta} \Bigl\{
[-\lambda^a(s)]\,\chi^a(s)\,v(s_0+\delta,\phi^a)+\chi^a(s)\,\ell^a(s)
\Bigr\}\,\dd s.
\end{split}
\end{equation}
Now we use the test function property of $\varphi$, i.e., we have
$(\varphi-w)(s_0,x_0,\tilde{x}_0)\le (\varphi-w)(s_0+\delta,x,\phi^a)$, which together
with  \eqref{E:value:supersol1} and
\eqref{E:value:supersol2} yield
\begin{align*}
&\int_{s_0}^{s_0+\delta} \left[\partial_t \varphi(s,x,\phi^a)+(x^\prime(s),\partial_x \varphi(s,x,\phi^a))
+(f^a(s),\partial_{\tilde{x}} \varphi(s,x,\phi^a))\right]\,\dd s\\
&=\varphi(s_0+\delta,x,\phi^a)-\varphi(s_0,x_0,\tilde{x}_0)\\
&\ge \left[u(s_0+\delta,x)-v(s_0+\delta,\phi^a)\right]-\left[y_0-\tilde{y}_0\right]\\
&> \int_{s_0}^{s_0+\delta} \left[(x^\prime(s),z)-H(s,x,y(s),z)
-\lambda^a(s)\,\chi^a(s)\,v(s_0+\delta,\phi^a)+\chi^a(s)\,\ell^a(s)\right]\,\dd s \\
&\qquad -\delta\,\eps.
\end{align*}
Consequently,
\begin{align*}
-\delta\,\eps&<\int_{s_0}^{s_0+\delta} \Bigl\{
 \partial_t \varphi(s,x,\phi^a)+H(s,x,y(s),z)
 -\left[\ell^a(s)+(f^a(s),\tilde{z})-\lambda^a(s)\,\tilde{y}_0\right]\\
 &\qquad\qquad\quad 
+(x^\prime(s),\partial_x \varphi(s,x,\phi^a)-z)
 +(f^a(s),\tilde{z}+\partial_{\tilde{x}} \varphi(s,x,\phi^a))\\
  &\qquad\qquad\quad 
  +[1-\chi^a(s)]\,\ell^a(s)
 +\lambda^a(s)\,\chi^a(s)\,v(s_0+\delta,\phi^a)
 -\lambda^a(s)\,\tilde{y}_0\Bigr\}\,\dd s\\
 &\le  
 \int_{s_0}^{s_0+\delta} \Bigl\{
 \partial_t \varphi(s,x,\phi^a)+H(s,x,y(s),z)
 -H(s,\phi^a,\tilde{y}_0,\tilde{z})\\
  &\qquad\qquad\quad 
+(x^\prime(s),\partial_x \varphi(s,x,\phi^a)-z)
 +(f^a(s),\tilde{z}+\partial_{\tilde{x}} \varphi(s,x,\phi^a))\Bigr\}\,\dd s\\
 &\qquad + I^a(s_0+\delta)+J^a(s_0+\delta),
\end{align*}
where
\begin{align*}
I^a(s_0+\delta)&:=\int_{s_0}^{s_0+\delta} [1-\chi^a(s)]\,\ell^a(s)\,\dd s
=\int_{s_0}^{s_0+\delta} \left[\int_{s_0}^s \lambda^a(r)\,\chi^a(r)\,\dd r\right]\,\ell^a(s)\,\dd s,\\
J^a(s_0+\delta)&:=\int_{s_0}^{s_0+\delta} \left[
\lambda^a(s)\,\chi^a(s)\,v(s_0+\delta,\phi^a)-\lambda^a(s)\,\tilde{y}_0\right]\,\dd s\\
&=\int_{s_0}^{s_0+\delta} \Bigl\{
\lambda^a(s)\,[\chi^a(s)-1]\,v(t,\phi^a)
+\lambda^a(s)\,v(s_0+\delta,\phi^a)-\lambda^a(s)\,\tilde{y}_0\Bigr\}\,\dd s\\ 
&=J^a_1(s_0+\delta)+J^a_2(s_0+\delta),\\
J^a_1(s_0+\delta)&:=
\int_{s_0}^{s_0+\delta}  \lambda^a(s)\,
\int_{s_0}^s [-\lambda^a(r)\,\chi^a(r)]\, \dd r\,\dd s\cdot v(s_0+\delta,\phi^a),\text{ and}\\
J^a_2(s_0+\delta)&:=
\int_{s_0}^{s_0+\delta}  \lambda^a(s)\,\dd s\cdot [v(s_0+\delta,\phi^a)-\tilde{y}_0].
\end{align*}
We estimate now those error terms.
By Assumption~\ref{A:DOC},
\begin{align*}
\abs{I^a(s_0+\delta)}&\le C_\lambda\,C_f\,(1+\sup\nolimits_{s\le s_0+\delta}\abs{\phi^a(s)})\,\delta^2\le C_1\,\delta^2
\end{align*}
for some constant $C_1$ that depends only on $\tilde{x}_0$ because $\phi^a\in\mathcal{X}^{(C_f)} (s_0,\tilde{x}_0)$.
By Assumption~\ref{A:DOC} and Proposition~\ref{P:value:regular},
\begin{align*}
\abs{J^a_1(s_0+\delta)}&\le C_\lambda^2\,\delta^2\,\abs{v(s_0+\delta,\phi^a)}\le C_2\,\delta^2\text{ and}\\
\abs{J^a_2(s_0,\delta)}&\le (C_\lambda\,\delta)\,\abs{v(s_0+\delta,\phi^a)-v(s_0,\phi^a)}\le C_2\,\delta^2
\end{align*}
for some constant $C_2$ that depends only on   $\tilde{x}_0$.
Therefore, for all sufficiently large $n\in\N$, we have, with $\tilde{x}^n:=\phi^{a^{\delta,\eps}}\vert_{\delta=\eps=n^{-1}}$,
\begin{equation}\label{E:value:supersol3}
\begin{split}
&-\frac{1}{n^2}-\frac{C_1+2C_2}{n^2}\\&\qquad\le
\int_{s_0}^{s_0+n^{-1}} \Bigl[
 \partial_t \varphi(s,x,\tilde{x}^n)+H(s,x,y(s),z)
 -H(s,\tilde{x}^n,\tilde{y}_0,\tilde{z})\\
 &\qquad\qquad 
 +(x^\prime(s),\partial_x \varphi(s,x,\tilde{x}^n)-z)
 +((\tilde{x}^n)^\prime(s),\tilde{z}+\partial_{\tilde{x}} \varphi(s,x,\tilde{x}^n))\Bigr]\,\dd s\\
 &\qquad\le 
 \int_{s_0}^{s_0+n^{-1}} \Bigl[\partial_t \varphi(s,x,\tilde{x}^n)+M_L(1+\abs{y(s)}+\abs{z})\sup_{t\le s} \abs{x(t)-\tilde{x}^n(t)}\\
&\qquad\qquad +L_H(1+\sup_{t\le s}\abs{\tilde{x}^n(t)})\abs{z-\tilde{z}} +0\\
&\qquad\qquad +(x^\prime(s),\partial_x\varphi(s,x,\tilde{x})-z)+
((\tilde{x}^n)^\prime(s),\partial_{\tilde{x}}\varphi(s,x,\tilde{x}^n)+\tilde{z})\Bigr]\,\dd s,
\end{split}
\end{equation}
where the last inequality can be derived exactly as in \eqref{L:Doubling:H2}.
Finally, let $\tilde{x}$ be a limit  in $\mathcal{X}^{(C_f)}(s_0,\tilde{x}_0)$ of a 
convergent subsequence $(\tilde{x}^{n_k})_k$ of $(\tilde{x}^n)_n$ 
such that also $((\tilde{x}^{n_k})^\prime)_k$ converges weakly to
$\tilde{x}^\prime$ in $L^2(s_0,T;\R^d)$. Replacing $n$ by $n_k$ in  \eqref{E:value:supersol3}
and letting $k\to\infty$ yields \eqref{E:Viscosity:H:updated2}.
This concludes the proof.
\end{proof}

\begin{lemma}\label{L:value:subsol}
Let $v$ be the value function defined by \eqref{E:value}.
Let $u$ be a minimax $C_f$-supersolution.
Then the function $(t,x,\tilde{x})\mapsto w(t,x,\tilde{x}):=v(t,x)-u(t,\tilde{x})$,
$[0,T]\times\Omega\times\Omega\to\R$, is a viscosity $C_f$-subsolution of 
\eqref{E:Doubled:H:updated2}  with parameter
$\Upsilon=v$.
\end{lemma}

\begin{proof}
{\color{black} First, note that  $v$ is continuous (Proposition~\ref{P:value:regular}) and  $u$ is l.s.c.
 Thus $w$ is u.s.c.~and
$(s,x,\tilde{x})\mapsto \Upsilon(s,x)-w(s,x,\tilde{x})=u(s,\tilde{x})$ is l.s.c.,
which is required by Definition~\ref{D:wuv:viscosity:updated2} in order for $w$ to
be a viscosity $C_f$-subsolution of \eqref{E:Doubled:H:updated2}.}

{\color{black} Next, let} $(s_0,x_0,\tilde{x}_0)\in [0,T)\times\Omega\times\Omega$.
Let $\varphi\in\underline{\mathcal{A}}^{(C_f)} w(s_0,x_0,\tilde{x}_0)$ with corresponding time
$T_0\in (s_0,T]$. Suppose that $w(s_0,x_0,\tilde{x}_0)>0$. Put
\begin{align*}
y_0&:=v(s_0,x_0),\quad \tilde{y}_0:=u(s_0,\tilde{x}_0),\quad
z:=\partial_x \varphi(s_0,x_0,\tilde{x}_0),\quad \tilde{z}:=-\partial_{\tilde{x}} \varphi(s_0,x_0,\tilde{x}_0).
\end{align*}
For every $a\in\mathcal{A}$, set
\begin{align*}
\phi^a&:=\phi^{s_0,x_0,a},\, \lambda^a:=\lambda^{s_0,x_0,a},\,
\chi^a:=\chi^{s_0,x_0,a},\, \ell^a:=\ell^{s_0,x_0,a},\,
f^a:=f(\cdot,\phi^a,a(\cdot)).
\end{align*}

Since $u$ is a minimax $C_f$-supersolution, there is a pair
$(\tilde{x},\tilde{y})\in\mathcal{Y}^{(C_f)}(s_0,\tilde{x}_0,\tilde{y}_0,\tilde{z})$ such that,
for all $t\in [s_0,T]$, we have
\begin{align}\label{E:value:subsol1}
u(t,\tilde{x})-\tilde{y}_0\le \tilde{y}(t)-\tilde{y}_0=\int_{s_0}^t \left[(\tilde{x}^\prime(s),\tilde{z})-
H(s,\tilde{x},\tilde{y}(s),\tilde{z})\right]\,\dd s.
\end{align}

{\color{black} Now,}
let $\delta\in (0,T_0-s_0]$ and $a\in\mathcal{A}$.
 By Proposition~\ref{P:DiscountedDPP},
 \begin{align*}
 y_0=v(s_0,x_0)\le \chi^a(s_0+\delta)\,v(s_0+\delta,\phi^a)+\int_{s_0}^{s_0+\delta} \chi^a(s)\,\ell^a(s)\,\dd s.
 \end{align*}
Thus
\begin{equation}\label{E:value:subsol2}
\begin{split}
v(s_0+\delta,\phi^a)-y_0&\ge [1-\chi^a(s_0+\delta)]\,v(s_0+\delta,\phi^a)-\int_{s_0}^{s_0+\delta}\chi^a(s)\,\ell^a(s)\,\dd s\\
&=\int_{s_0}^{s_0+\delta}\left[
\lambda^a(s)\,\chi^a(s)\,v(s_0+\delta,\phi^a)-\chi^a(s)\,\ell^a(s)\right]\,\dd s.
\end{split}
\end{equation}
By \eqref{E:value:subsol1}, \eqref{E:value:subsol2}, and the test function property of $\varphi$,
\begin{align*}
&\int_{s_0}^{s_0+\delta}\left[
\partial_t \varphi(s,\phi^a,\tilde{x})+(f^a(s),\partial_x \varphi(s,\phi^a,\tilde{x}))+
(\tilde{x}^\prime(s),\partial_{\tilde{x}} \varphi(s,\phi^a,\tilde{x}))\right]\,\dd s\\
&\quad \ge [v(s_0+\delta,\phi^a)-u(s_0+\delta,\tilde{x})]-[y_0-\tilde{y}_0]\\
&\quad \ge \int_{s_0}^{s_0+\delta}\left[\lambda^a(s)\,\chi^a(s)\,v(s_0+\delta,\phi^a)-\chi^a(s)\,\ell^a(s)
-(\tilde{x}^\prime(s),\tilde{z})+H(s,\tilde{x},\tilde{y}(s),\tilde{z})\right]\,\dd s.
\end{align*}
Therefore,
\begin{equation}\label{E:value:subsol3}
\begin{split}
&\int_{s_0}^{s_0+\delta} \Bigl\{\partial_t \varphi(s,\phi^a,\tilde{x})
+[\ell^a(s)+(f^a(s),z)-\lambda^a(s)\,y_0]
-H(s,\tilde{x},\tilde{y}(s),\tilde{z})\\
&\qquad\qquad +(f^a(s),\partial_x\varphi(s,\phi^a,\tilde{x})-z)
+(\tilde{x}^\prime(s),\partial_{\tilde{x}}\varphi(s,\phi^a,\tilde{x})+\tilde{z})\Bigr\}\,\dd s\\
&\quad\ge \int_{s_0}^{s_0+\delta} \Bigl\{\lambda^a(s)\,\chi^a(s)\,v(s_0+\delta,\phi^a)
-\lambda^a(s)\,y_0+(1-\chi^a)\,\ell^a(s)\Bigr\}\,\dd s\\
&\quad\ge -C\delta^2
\end{split}
\end{equation}
for some constant $C>0$ independent from $\delta$ and $a$. The last inequality in \eqref{E:value:subsol3}
can be shown exactly as the estimation of the terms $ I^a(s_0+\delta)$ and $J^a(s_0+\delta)$
in the proof of Lemma~\ref{L:value:supersol}.

We continue now by proceeding  similarly as in the Step~1 of  proof of Theorem~6.7 in \cite{Qiu18SICON_SHJB},
i.e., we obtain, via a measurable selection argument (e.g., Theorem~21.3.4 in \cite{CohenElliott}),
for every $\eps>0$
the existence of a control $a^\eps\in\mathcal{A}$ such that
\begin{align*}
\ell^{a^\eps}(s)+(f^{a^\eps}(s),z)-\lambda^{a^\eps}(s)\,y_0\le H(s,\phi^{a^\eps},y_0,z)+\eps\quad\text{a.e.~in $(s_0,T_0)$.}
\end{align*}
Thus, together with \eqref{E:value:subsol3},
\begin{equation}\label{E:valuesubsol4}
\begin{split}
&-C\delta^2-\delta\eps\le \int_{s_0}^{s_0+\delta} \Bigl[
\partial_t \varphi(s,\phi^{a^\eps},\tilde{x})
+H(s,\phi^{a^\eps},y_0,z)-H(s,\tilde{x},\tilde{y}(s),z)\\
&\qquad\qquad\qquad +(f^{a^\eps}(s),\partial_x\varphi(s,\phi^{a^\eps},\tilde{x})-z)
+(\tilde{x}^\prime(s),\partial_{\tilde{x}}\varphi(s,\phi^{a^\eps},\tilde{x})+\tilde{z})
\Bigr]\,\dd s\\
&\qquad\le  \int_{s_0}^{s_0+\delta} \Bigl[\partial_t \varphi(s,\phi^{a^\eps},\tilde{x})+
M_L(1+\abs{y_0}+\abs{z})\sup_{t\le s} \abs{\phi^{a^{\eps}}(t)-\tilde{x}(t)}\\
&\qquad\qquad\qquad +L_H(1+\sup_{t\le s}\abs{\tilde{x}(t)})\abs{z-\tilde{z}}\\
&\qquad\qquad\qquad +(f^{a^\eps}(s),\partial_x\varphi(s,\phi^{a^\eps},\tilde{x})-z)
+(\tilde{x}^\prime(s),\partial_{\tilde{x}}\varphi(s,\phi^{a^\eps},\tilde{x})+\tilde{z})\Bigr]\,\dd s.
\end{split}
\end{equation}
We refer to \eqref{L:Doubling:H2} for details regarding the last inequality.
Finally, dividing \eqref{E:valuesubsol4} by $\delta$, letting $\delta\downarrow 0$, and noting
that $\eps>0$ was arbitrary yields \eqref{E:Viscosity:H:updated2}, which concludes the proof.
\end{proof}

\begin{theorem}\label{T:DiscoundedVerification}
The value function $v$ is the unique minimax $C_f$-solution of \eqref{E:DiscountedHJB}.
\end{theorem}
\begin{proof}
By Theorem~\ref{T:Perron}, there is a unique minimax  $C_f$-solution $u$ of of \eqref{E:DiscountedHJB}.
By Theorem~\ref{T:Comparison:H} and Lemma~\ref{L:value:supersol}, $u\le v$.
By Theorem~\ref{T:Comparison:H} and Lemma~\ref{L:value:subsol}, $v\le u$.
\end{proof}

{\color{black}
\section{The non-path-dependent case}\label{S:NonPPDE}
\subsection{General theory}
Consider the non-path-dependent counterpart of \eqref{E:PPDE:H},
i.e., the terminal value problem
\begin{equation}\label{E:PPDE:widehatH}
\begin{split}
-\frac{\partial}{\partial t} \widehat{u}(t,\xi)-\widehat{H}\Big (t,\xi,\widehat{u}(t,\xi),
\frac{\partial}{\partial\xi} \widehat{u}(t,\xi)\Big )&=0,\quad (t,\xi)\in (0,T)\times\R^d,\\
 \widehat{u}(T,\xi)&=\widehat{h}(\xi),\quad \xi\in\R^d,
\end{split}
\end{equation}
where $\widehat{H}:[0,T]\times\R^d\times\R\times\R^d\to\R$ and $\widehat{h}:\R^d\to\R$.
Minimax solutions for \eqref{E:PPDE:widehatH} 
with continuous Hamiltonian $\widehat{H}$ have been studied in \cite{Subbotin}.
Here, we consider the case of $\widehat{H}$ to be only measurable in time.

The following assumptions are in force.

\begin{assumption}\label{A:widehat:h}
The function $\widehat{h}$ is continuous.
\end{assumption}

\begin{assumption}\label{A:widehat:H}
The following holds:

(i) For a.e.~$t\in (0,T)$, the function $(\xi,y,z)\mapsto \widehat{H}(t,\xi,y,z)$, $\R^d\times\R\times\R^d\to\R$,
is continuous.

(ii) For each $(\xi,y,z)\in\R^d\times\R\times\R^d$, the function $t\mapsto \widehat{H}(t,\xi,y,z)$, $[0,T]\to\R$,
is Borel measurable.

(iii) There is a constant $L_{\widehat{H}}\ge 0$ 
such that, for a.e.~$t\in (0,T)$, every $\xi\in\R^d$, $y\in\R$, $z$, $\tilde{z}\in\R^d$, we have
\begin{align*}
|\widehat{H}(t,\xi,y,z)-\widehat{H}(t,\xi,y,\tilde{z})|
\le L_{\widehat{H}}(1+\abs{\xi})\abs{z-\tilde{z}}.
\end{align*}

(iv)  There is a constant $M_{\widehat{H}}$ such that 
for every $\xi$, $\tilde{\xi}\in\R^d$,
$y\in\R$, $z\in\R^d$, and a.e.~$t\in (t_0,T)$,
\begin{align*}
|\widehat{H}(t,\xi,y,z)-\widehat{H}(t,\tilde{\xi},y,z)|&\le M_{\widehat{H}}
(1{+\abs{y}}+\abs{z})\,\sup_{s\le t} |\xi-\tilde{\xi}|. 
\end{align*}

(v) For a.e.~$t\in (0,T)$ and each $(\xi,z)\in\R^d\times\R^d$, the function 
$y\mapsto \widehat{H}(t,\xi,y,z)$, $\R\to\R$, is non-increasing.

(vi) There exists  a constant $C_{\widehat{H}}\ge 0$ such that, for a.e.~$t\in (0,T)$  and each $(\xi,y)\in\R^d\times\R$,
we have
\begin{align*}
|\widehat{H}(t,\xi,y,0)|\le C_{\widehat{H}}(1+\abs{\xi}+\abs{y}).
\end{align*}
\end{assumption}

Next, we define function spaces needed for the definition of minimax solutions in the non-path-dependent case.
Given $L\ge 0$,  $s_0\in [0,T)$, $\xi\in\R^d$, $y_0\in\R$, and $z\in\R^d$, put
\begin{align*}
&\widehat{\mathcal{X}}^L(s_0,\xi):=\Big \{\text{all absolutely continuous functions $x:[s_0,T]\to\R^d$ 
with}\\
&\qquad\qquad \abs{x^\prime(t)}\le L\Big (1+\sup_{s_0\le s\le t} \abs{x(s)}\Big )\text{ a.e.~on $(s_0,T)$
and $x(s_0)=\xi$}
\Big \},\\
&\widehat{\mathcal{Y}}^L
(s_0,\xi,y_0,z):=\Big \{(x,y)\in\widehat{\mathcal{X}}^L(s_0,\xi)\times C([s_0,T]):\\
&\qquad\qquad y(t)=y_0+\int_{s_0}^t [{\color{black} x^\prime(s)\cdot z}-
\widehat{H}(s,x(s),y(s),z)]\,\dd s\text{ on $[s_0,T]$}\Big \}.
\end{align*}

\begin{definition}\label{D:widehatMinimaxSolution}
Let $L\ge 0$ and $\widehat{u}:[0,T]\times\R^d\to\R$.

(i) $\widehat{u}$ 
is a \emph{minimax $L$-supersolution} of \eqref{E:PPDE:widehatH} 
if $\widehat{u}\in\mathrm{LSC}([0,T]\times\R^d)$,
if $\widehat{u}(T,\cdot)\ge \widehat{h}$ on $\R^d$, and if, for every $(s_0,\xi,z)\in [0,T)\times\R^d\times\R^d$, and
every $y_0\ge \widehat{u}(s_0,\xi)$, there exists an $(x,y)\in\widehat{\mathcal{Y}}^L(s_0,\xi,y_0,z)$ such that
$y(t)\ge \widehat{u}(t,x(t))$ for each $t\in [s_0,T]$.

(ii)  $\widehat{u}$ is a \emph{minimax $L$-subsolution} of 
\eqref{E:PPDE:widehatH} if $\widehat{u}\in\mathrm{USC}([0,T]\times\R^d)$,
if $\widehat{u}(T,\cdot)\le \widehat{h}$ on $\R^d$, and if, for every $(s_0,\xi,z)\in [0,T)\times\R^d\times\R^d$, and
every $y_0\le \widehat{u}(s_0,\xi)$, there exists an $(x,y)\in\widehat{\mathcal{Y}}^L(s_0,\xi,y_0,z)$ such that
$y(t)\le \widehat{u}(t,x(t))$ for each  $t\in [s_0,T]$.

(iii) $\widehat{u}$ is a \emph{minimax $L$-solution} of \eqref{E:PPDE:widehatH} if it is 
both a \emph{minimax $L$-supersolution}
and a \emph{minimax $L$-subsolution} of  \eqref{E:PPDE:widehatH}.
\end{definition}

Note that the above definition is due to \cite{Subbotin}.
We want to remark that in \cite{Subbotin} instead of $\widehat{\mathcal{X}}^L(s_0,\xi)$ the smaller space
\begin{align*}
&\{\text{all absolutely continuous functions $x:[s_0,T]\to\R^d$ 
with}\\
&\qquad\qquad \abs{x^\prime(t)}\le L(1+\ \abs{x(t)})\text{ a.e.~on $(s_0,T)$
and $x(s_0)=\xi$}\}
\end{align*}
is used.

We establish now consistency results between the non-path-dependent case in this section
and the path-dependent case treated previously
(note that this topic has been 
addressed in section~4.1 of \cite{GLP21AMO} and by the discussion on p.~276 in \cite{GL24survey}).
Moreover, we immediately obtain a comparison principle for \eqref{E:PPDE:widehatH}.
\begin{proposition}\label{P:widehat:consistency}
Let   $h:\Omega\to\R$ and $H:[0,T]\times\Omega\times\R\times\R^d\to\R$ be  defined by.
\begin{align*}
h(x):=\widehat{h}(x(T))\text{ and }  H(t,x,y,z):=\widehat{H}(t,x(t),y,z)
\end{align*}
Let $L\ge 0$. Then the following holds:

(i) If $\widehat{u}$ is a minimax $L$-supersolution of \eqref{E:PPDE:widehatH},
then $u:[0,T]\times\Omega\to\R$ defined by $u(t,x):=\widehat{u}(t,x(t))$ is
a minimax $L$-supersolution of  \eqref{E:PPDE:H}.

(ii) If $\widehat{u}$ is a minimax $L$-subsolution of \eqref{E:PPDE:widehatH},
then $u:[0,T]\times\Omega\to\R$ defined by $u(t,x):=\widehat{u}(t,x(t))$ is
a minimax $L$-subsolution of  \eqref{E:PPDE:H}.

(iii) If $\widehat{u}_1$ is a minimax $L$-sub- and 
 $\widehat{u}_2$ is a minimax $L$-supersolution of \eqref{E:PPDE:widehatH},
 then $\widehat{u}_1\le \widehat{u}_2$.
 
 (iv) If $u$ is a a minimax $L$-supersolution of  \eqref{E:PPDE:H} 
 and $u(t,x)=\widehat{u}(t,x(t))$ for some function $\widehat{u}:[0,T]\times\R^d\to\R$,
 then $\widehat{u}$ is a minimax $L$-supersolution of \eqref{E:PPDE:widehatH}.
 
  (v) If $u$ is a a minimax $L$-subsolution of  \eqref{E:PPDE:H}
 and $u(t,x)=\widehat{u}(t,x(t))$ for some function $\widehat{u}:[0,T]\times\R^d\to\R$,
 then $\widehat{u}$ is a minimax $L$-subsolution of \eqref{E:PPDE:widehatH}.
\end{proposition}

\begin{proof}
First note that, thanks to  Assumptions~\ref{A:widehat:h} and \ref{A:widehat:H},
the functions $h$ and $H$ satisfy Assumptions~\ref{A:h} and \ref{A:H}, i.e.,
we can apply the results of section 
 \ref{S:Comparison}.

We prove only (i), (iii), and (iv). The proofs of (ii) and (v), resp., are parallel to the proofs of (i) and (iv), resp.

(i) Clearly, $\widehat{u}\in\mathrm{LSC}([0,T]\times\R^d)$ implies $u\in\mathrm{LSC}([0,T]\times\Omega)$
and we also have $u(T,\cdot)\ge h$ on $\Omega$. 

It remains to check the interior condition. To this end,
fix $(s_0,x_0,z)\in [0,T)\times\Omega\times\R^d$ and $y_0\ge u(s_0,x_0)$. Since $\widehat{u}$
is a minimax  $L$-supersolution of \eqref{E:PPDE:widehatH},
there is  a pair $(x,y)\in\widehat{\mathcal{Y}}^L(s_0,x_0(s_0),y_0,z)$ such that, for every $t\in [s_0,T]$, we have
$y(t)\ge \widehat{u}(t,x(t))$.
Define  $\tilde{x}\in\Omega$ by 
$$\tilde{x}(t):=\bfone_{[0,s_0)}(t)\,x_0(t)+\bfone_{[s_0,T]}(t)\,x(t).
$$
Then $(\tilde{x},y)\in \mathcal{Y}^L(s_0,x_0,y_0,z)$ and, for every $t\in [s_0,T]$, we have
$y(t)\ge u(t,\tilde{x})$, i.e., we have established that $u$ is a minimax $L$-supersolution of  \eqref{E:PPDE:H}.

(iii) By parts~(i) and (ii) together with Corollary~\ref{C:Comparison}, we have  $\widehat{u}_1\le \widehat{u}_2$.

(iv) It is easy to check that
 $u\in\mathrm{LSC}([0,T]\times\Omega)$ implies $\widehat{u}\in\mathrm{LSC}([0,T]\times\R^d)$ 
and we also have $\widehat{u}(T,\cdot)\ge \widehat{h}$ on $\R^d$.

To verify the interior condition, let $(s_0,\xi,z)\in [0,T)\times\R^d\times\R^d$ and
$y_0\ge \widehat{u}(s_0,\xi)$.
Consider the constant path $x_0\in\Omega$ defined by 
$$
x_0(t):=\xi.
$$
Since $u$ is a minimax $L$-supersolution of  \eqref{E:PPDE:H},
there is a pair $(x,y)\in\mathcal{Y}^L(s_0,x_0,y_0,z)$ such that, for every $t\in [s_0,T]$,
we have 
$$y(t)\ge u(t,x)=\widehat{u}(t,x(t)).
$$
Since $(x\vert_{[s_0,T]},y)\in\widehat{\mathcal{Y}}^L(s_0,\xi,y_0,z)$,
we can conclude that $\widehat{u}$ is a minimax $L$-supersolution of \eqref{E:PPDE:widehatH}.
\end{proof}
\subsection{Optimal control}
We use as control space the set $A$ and as set of admissible controls the set $\mathcal{A}$,
which have been introduced at the beginning of section~\ref{S:OptimalControl}. The remaining data
are  functions $\widehat{f}:[0,T]\times\R^d\times A\to\R^d$, $\widehat{\lambda}:[0,T]\times\R^d\times A\to \R_+$,
$\widehat{\ell}:[0,T]\times\R^d\times A\to\R$, and $\widehat{h}:\R^d\to\R$.

The next assumption is in force.

\begin{assumption}\label{A:widehat:DOC}
Suppose that the following holds:

(i) For a.e.~$t\in [0,T]$, 
the map $(\xi,a)\mapsto (\widehat{f},\widehat{\lambda},\widehat{\ell})(t,\xi,a)$, $\R^d\times A\to\R^d\times\R_+\times\R$,
is continuous. 

(ii) For all $(\xi,a)\in\R^d\times A$, the map $t\mapsto
 (\widehat{f},\widehat{\lambda},\widehat{\ell})(t,\xi,a)$, $[0,T]\to\R^d\times\R_+\times\R$,
is Borel measurable.

(iii) There are constants $C_{\widehat{f}}$, $C_{\widehat{\lambda}}\ge 0$ such that,
 for a.e.~$t\in (0,T)$ and all $\xi\in\R^d$,  we have
\begin{align*}
\sup_{a\in A}\left(|\widehat{f}(t,\xi,a))|+|\widehat{\ell}(t,\xi,a)|\right)\le 
C_{\widehat{f}}\,(1+\abs{\xi})
\text{ and } \sup_{a\in A} \abs{\lambda(t,\xi,a)}\le C_{\widehat{\lambda}}.
\end{align*}

(iv) There is a constant $L_{\widehat{f}}\ge 0$ such that, for a.e.~$t\in (0,T)$ and all $\xi$, $\tilde{\xi}\in\R^d$, 
we have
\begin{align*}
&\sup_{a\in A}\left(|\widehat{f}(t,\xi,a)-\widehat{f}(t,\tilde{\xi},a)|+
|\widehat{\ell}(t,\xi,a)-\widehat{\ell}(t,\tilde{\xi},a)|+
|\widehat{\lambda}(t,\xi,a)-\widehat{\lambda}(t,\tilde{\xi},a)|\right)\\ 
&\qquad +|\widehat{h}(\xi)-\widehat{h}(\tilde{\xi})| \le L_{\widehat{f}}\,|\xi-\tilde{\xi}|.
\end{align*}
\end{assumption}

Given $(s,\xi,a)\in [0,T)\times\R^d\times \mathcal{A}$, denote by
$\phi^{s,\xi,a}$ the solution of 
\begin{align*}
(\phi^{s,\xi,a})^\prime(t)&=\widehat{f}(t,\phi^{s,\xi,a}(t),a(t))\text{ a.e.~on $(s,T)$}\\
\phi^{s,\xi,a}(s)&=\xi.
\end{align*}

Our optimal control problem is implicitly specified by 
the value function  $\widehat{v}:[0,T]\times\R^d\to\R$  
that is defined by
\begin{equation}\label{E:widehat:value}
\begin{split}
\widehat{v}(s,\xi)&:=\inf_{a\in\mathcal{A}} 
\Big[\int_s^T 
\exp\Big (-\int_s^t \widehat{\lambda}(r,\phi^{s,\xi,a}(r),a(r))\,\dd r\Big )
\,\widehat{\ell}(t,\phi^{s,\xi,a}(t),a(t))\,\dd t\\ &\qquad\qquad + 
\exp\Big (-\int_s^T \widehat{\lambda}(r,\phi^{s,\xi,a}(r),a(r))\,\dd r\Big )\,\widehat{h}(\phi^{s,\xi,a}(T))\Big].
\end{split}
\end{equation}

The corresponding HJB equation is
\begin{equation}\label{E:widehat:DiscountedHJB}
\begin{split}
-\frac{\partial}{\partial t} \widehat{u}(t,\xi)-
\widehat{H}\Big (t,\xi,\widehat{u}(t,\xi),
\frac{\partial}{\partial \xi} \widehat{u}(t,\xi)
\Big )
 &=0, 
 \quad (t,\xi)\in (0,T)\times\R^d,\\
\widehat{u}(T,\xi)&=\widehat{h}(\xi),\quad \xi\in\R^d,
\end{split}
\end{equation}
where $\widehat{H}:[0,T]\times\R^d\times\R\times\R^d\to\R$ is defined by
\begin{align*}
\widehat{H}(t,\xi,y,z):=
\inf_{a\in A} [\widehat{\ell}(t,\xi,a)+(\widehat{f}(t,\xi,a),z)-
\widehat{\lambda}(t,\xi,a)\,y].
\end{align*}

\begin{theorem}
The value function $\widehat{v}$ is the unique $C_{\widehat{f}}\,$-minimax solution of 
\eqref{E:widehat:DiscountedHJB}.
\end{theorem}
\begin{proof}
Define $v:[0,T]\times\Omega\to\R$ by
\begin{align*}
v(t,x):=\widehat{v}(t,x(t)).
\end{align*}
Also define $(f,\lambda,\ell):[0,T]\times\Omega\times A\to\R^d\times\R_+\times\R$ and $h:\Omega\to\R$ by
\begin{align}\label{E:widehat:data}
(f,\lambda,\ell)(t,x,a):=(\widehat{f}, \widehat{\lambda},\widehat{\ell})(t,x(t),a)\text{ and }
h(x):=\widehat{h}(x(T)).
\end{align}
Then $v$ is the value function of the path-dependent optimal control problem in section~\ref{S:OptimalControl}
with data specified by \eqref{E:widehat:data}, i.e., $v$ satisfies \eqref{E:value}.
Note that  $(f,\lambda,\ell)$ and $h$ satisfy Assumption~\ref{A:DOC}
with $C_f=C_{\widehat{f}}$,
$C_\lambda=C_{\widehat{\lambda}}$, and 
$L_f=L_{\widehat{f}}$ thanks to Assumption~\ref{A:widehat:DOC}.
Thus, by Theorem~\ref{T:DiscoundedVerification}, 
$v$ is the unique minimax $C_f$-solution of \eqref{E:DiscountedHJB}
with $H:[0,T]\times\Omega\times\R\times\R^d\to\R$ defined by
\begin{align*}
H(t,x,y,z):=\widehat{H}(t,x(t),y,z).
\end{align*}
Hence, by Proposition~\ref{P:widehat:consistency} (iv) and (v),
the value function $\widehat{v}$ is a minimax $C_{\widehat{f}}$\,-solution
of \eqref{E:widehat:DiscountedHJB}.
Finally,  the comparison principle in form of Proposition~\ref{P:widehat:consistency} (iii) yields
uniqueness.  This concludes the proof.
\end{proof}
}

\appendix
{\color{black}
\section{On the regularity of the value function in section~\ref{S:OptimalControl}}\label{S:Appendix:Regularity}
We establish the regularity of the value function $v$ defined in \eqref{E:value}.
Our approach is based on section~7 of \cite{BK18JFA},
especially on the proofs of Proposition~7.2
and Theorem~7.4 therein.
Compared to \cite{BK18JFA}, there are slight differences due to the discount factors.

We freely use here the data of the optimal control problem
in section~\ref{S:OptimalControl}.
Moreover, Assumption~\ref{A:DOC} is in force.

We start with a basic estimate concerning $\mathcal{X}^L(t_0,x_0)$.

\begin{lemma}\label{L:Appendix:XL}
Let  $(t_0,x_0)\in [0,T]\times\Omega$, $L\ge 0$, $x\in\mathcal{X}^L(t_0,x_0)$, and $t\in [t_0,T]$. Then
 \begin{align}\label{E:L:Appendix:XL}
 \norm{x(\cdot\wedge t)}_\infty\le (LT+\norm{x_0(\cdot\wedge t_0)}_\infty)\,e^{L(t-t_0)}.
 \end{align}
\end{lemma}
\begin{proof}
By the definition of $\mathcal{X}^L(t_0,x_0)$
in \eqref{E:XLs0x0}, we can infer that, for each $s\in [t_0,T]$,
\begin{align*}
 \norm{x(\cdot\wedge s)}_\infty\le \norm{x(\cdot\wedge t_0)}_\infty+\int_{t_0}^s L(1+ \norm{x(\cdot\wedge r)}_\infty)\,\dd r.
\end{align*}
Gronwall's inequality yields \eqref{E:L:Appendix:XL}.
\end{proof}

Next, we establish regularity of $\phi^{t,x,a}$ in $x$ as well as in $t$.
\begin{lemma}\label{L:Appendix:Reg:x}
Fix $t_0\in [0,T)$, $x_0$, $\tilde{x}_0\in\Omega$, and $a\in\mathcal{A}$.
Then, for all $t\in [t_0,T]$, 
\begin{align}\label{E:L:Appendix:Reg:x}
\abs{\phi^{t_0,x_0,a}(t)-\phi^{t_0,\tilde{x}_0,a}(t)}\le e^{L_f\, (t-t_0)}\,\sup_{r\le t_0} 
\abs{x_0(r)-\tilde{x}_0(r)}.
\end{align}
\end{lemma}
\begin{proof}
Let $t\in [t_0,T]$
and put $\Delta\phi(s):=\phi^{t_0,x_0,a}(s)-\phi^{t_0,\tilde{x}_0,a}(s)$
for each $s\in [0,T]$. By \eqref{E:phi:prime:f} and 
Assumption~\ref{A:DOC}, we have
\begin{align*}
\sup_{r\le t} \abs{\Delta \phi(r)}&\le
\sup_{r\le t_0} \abs{x_0(r)-\tilde{x}_0(r)}+\int_{t_0}^t
\abs{f(s,\phi^{t_0,x_0,a},a(s))-f(s,\phi^{t_0,\tilde{x}_0,a},a(s))}\,\dd s\\
&\le \sup_{r\le t_0} \abs{\Delta \phi(r)}+\int_{t_0}^t
L_f\, \sup_{r\le s} \abs{\Delta \phi(r)}\,\dd s.
\end{align*}
By Gronwall's inequality, we immediately have \eqref{E:L:Appendix:Reg:x}.
\end{proof}

\begin{lemma}\label{L:Appendix:Reg:t1}
Let $0\le t_0\le  t_1\le T$, $x_0\in\Omega$, $a\in\mathcal{A}$,
$L\ge 0$, and $x\in\mathcal{X}^L(t_0,x_0)$.
Then,  for every $t\in [t_1,T]$, we have
\begin{align*}
&\sup_{s\le t} \abs{\phi^{t_0,x_0,a}(s)-\phi^{t_1,x,a}(s)}\\&\quad\le
(C_f+L)\Big [1+
\Big (
e^{C_f(t-t_0)}+e^{L(t-t_0)}
\Big )
\Big (C_fT+LT+\norm{x_0(\cdot\wedge t_0)}_\infty\Big )\Big ]
\, (t_1-t_0).
\end{align*}
\end{lemma}
\begin{proof}
Let $t\in [t_1,T]$. By Lemma~\ref{L:Appendix:Reg:x}, 
\begin{align*}
\abs{
\phi^{t_0,x_0,a}(t)-\phi^{t_1,x,a}(t)
}&=|
\phi^{t_1,(\phi^{t_0,x_0,a}),a}(t)-\phi^{t_1,x,a}(t)
|\\
&\le e^{L_f (t-t_1)}\,\sup_{r\le t_1}\abs{
\phi^{t_0,x_0,a}(r)-x(r)
}.
\end{align*}
Since, for every $r\in [t_0,t_1]$, by Assumption~\ref{A:DOC},
\begin{align*}
&\abs{
\phi^{t_0,x_0,a}(r)-x(r)
}=\abs{
\int_{t_0}^r f(\theta,\phi^{t_0,x_0,a},a(\theta))\,\dd \theta
-\int_{t_0}^r x^\prime(\theta)\,\dd \theta
}\\
&\le (r-t_0)\, C_f\, \Big (1+\sup_{\theta\le r} \abs{\phi^{t_0,x_0,a}(\theta)}\Big )+(r-t_0)\, L\, \Big (1+\sup_{\theta\le r} \abs{x(\theta)}\Big )\\
&\le (r-t_0)\, (C_f+L)\, \Big [1+
(e^{C_f(r-t_0)}+e^{L(r-t_0)}) \Big (C_fT+LT+\sup_{\theta\le t_0} \abs{x_0(\theta)}\Big )\Big ],
\end{align*}
where the last inequality follows from 
Lemma~\ref{L:Appendix:XL} together with  noting that we have
 $\phi^{t_0,x_0,a}\in\mathcal{X}^{C_f}(t_0,x_0)$ due to Assumption~\ref{A:DOC}.
This concludes the proof.
\end{proof}

\begin{proof}[Proof of Proposition~\ref{P:value:regular}]
We  start the proof by establishing two claims concerning the value function $v$
defined in \eqref{E:value}.\newline

\noindent\textit{Claim 1: There is constant $C_1>0$ such that,
for all $(s,x,\tilde{x})\in [0,T]\times\Omega\times\Omega$,} 
\begin{align}\label{E:Appendix:Claim1:1}
\abs{v(s,x)-v(s,\tilde{x})}\le  C_1(1+\norm{\tilde{x}(\cdot\wedge s)}_\infty)\, \norm{x(\cdot\wedge s)-\tilde{x}(\cdot\wedge s)}_\infty.
\end{align}

\noindent\textit{Proof of Claim 1.} Fix  $(s,x,\tilde{x})\in [0,T]\times\Omega\times\Omega$ and $a\in\mathcal{A}$. Put
\begin{align*}
(\chi,\lambda,\ell,\phi)&:=(\chi^{s,x,a},\lambda^{s,x,a},\ell^{s,x,a},\phi^{s,x,a}),\\
(\tilde{\chi},\tilde{\lambda},\tilde{\ell},\tilde{\phi})&:=(\chi^{s,\tilde{x},a},\lambda^{s,\tilde{x},a},\ell^{s,\tilde{x},a},\phi^{s,\tilde{x},a}),\\
(\Delta\chi,\Delta\ell, \Delta h,\Delta x)&:=(\chi-\tilde{\chi},\ell-\tilde{\ell},h(\phi)-h(\tilde{\phi}), \norm{x(\cdot\wedge s)-\tilde{x}(\cdot\wedge s)}_\infty).
\end{align*}
We  show that there is a constant $C_1>0$ independent of $s$, $x$, $\tilde{x}$, and $a$ such that
\begin{equation*}
\begin{split}
\Delta J&:=\Big [\int_s^T \chi(t)\ell(t)\,\dd t+\chi(T)h(\phi)\Big ]
-\Big[\int_s^T \tilde{\chi}(t)\tilde{\ell}(t)\,\dd t+\tilde{\chi}(T)h(\tilde{\phi})\Big]\\
&\le C_1(1+\norm{\tilde{x}(\cdot\wedge s)}_\infty)\Delta x,
\end{split}
\end{equation*}
which in turn will  imply \eqref{E:Appendix:Claim1:1}. 
To this end, we follow the proof of Lemma~5.6 of \cite{BK2},
which is a  slightly more involved version of Claim~1. As in
\cite[p.~31]{BK2}, we have
\begin{align*}
\Delta J=I_1+I_2,
\end{align*}
where
\begin{align*}
I_1&:=\chi(T)\,\Delta h+ \Delta\chi(T)\,h(\tilde{\phi}),\\
I_2&:=\int_s^T \chi(t)\,\Delta\ell(t)+\Delta\chi(t)\,\tilde{\ell}(t)\,\dd t
\end{align*}
and, thanks to Lemmas~\ref{L:Appendix:Reg:x} and \ref{L:Appendix:XL}, Assumption~\ref{A:DOC}, 
\begin{align}\label{E:chi:Davis}
\Delta\chi(t)\le L_f(t-s)e^{L_f(t-s)}\Delta x,
\quad t\in [s,T],\qquad
\text{(see p.~174 in \cite{DavisBook})}
\end{align}
as well as noting that $\abs{h(\tilde{\phi})}\le L_f\norm{\tilde{\phi}}_\infty+\abs{h(\mathbf{0})}$, we have
\begin{align*}
I_1&\le L_f e^{L_f T}\Delta x+\left(L_f Te^{L_f T} \Delta x\right)\,
\left[
L_f (C_f T+\norm{\tilde{x}(\cdot\wedge s)}_\infty) e^{C_f T}+\abs{h(\mathbf{0})} 
\right],  \\
I_2&\le  L_f T e^{L_f T}\Delta x+ \left(L_f T^2 e^{L_f T}\Delta x\right)\,
\left[
C_f+ C_f^2 T e^{LT}+C_f e^{LT}\norm{\tilde{x}(\cdot\wedge s)}_\infty
\right].
\end{align*}
Here, $\mathbf{0}$ is the zero path in $\Omega$.
This concludes the proof of Claim~1.\newline

\noindent\textit{Claim 2:  Given $L\ge 0$,  
there is a constant $C_2>0$ such that $x_0\in\Omega$,
 $0\le s\le t\le T$ and $x\in\mathcal{X}^L(s,x_0)$ imply}
\begin{align}\label{E:Appendix:Claim2:1}
\abs{v(s,x)-v(t,x)}\le  C_2(1+\norm{x_0(\cdot\wedge s)}_\infty)^2\,(t-s).
\end{align}

\noindent\textit{Proof of Claim 2.}
Fix $L\ge 0$, $x_0\in\Omega$, and $a\in\mathcal{A}$.  Let $0\le s\le t\le T$ and $x\in\mathcal{X}^L(s,x_0)$. Put
\begin{align*}
(\chi^s,\lambda^s,\ell^s,\phi^s)&:=(\chi^{s,x,a},\lambda^{s,x,a},\ell^{s,x,a},\phi^{s,x,a}),\\
(\chi^t,\lambda^t,\ell^t,\phi^t)&:=(\chi^{t,x,a},\lambda^{t,x,a},\ell^{t,x,a},\phi^{t,x,a}),\\
(\Delta\chi^s,\Delta\ell^s, \Delta h^s)&:=
(\chi^s-\chi^t,\ell^s-\ell^t,h(\phi^s)-h(\phi^t)).
\end{align*}
We 
show that there is a constant $C_2>0$ independent of $s$, $t$, $x_0$, $x$, and $a$ such
that
\begin{align}\label{E:Appendix:Claim2:2}
\abs{\Delta J^s} \le C_2(1+\norm{x_0(\cdot\wedge s)}_\infty)^2\,(t-s),
\end{align}
where
\begin{align*}
\Delta J^s:=\Big[
\int_s^T \chi^s(r)\,\ell^s(r)\, \dd r +\chi^s(T) h(\phi^s)\Big] 
-\Big[
\int_t^T \chi^t(r)\,\ell^t(r)\, \dd r +\chi^t(T) h(\phi^t)\Big].
\end{align*}
This will immediately yield \eqref{E:Appendix:Claim2:1}.
To show \eqref{E:Appendix:Claim2:2}, we proceed similarly as in the proof of Claim~1.
First note that, by Assumption~\ref{A:DOC} and Lemma~\ref{L:Appendix:XL},  we have
\begin{align*}
\abs{\Delta J^s}&\le \abs{\int_s^t \chi^s(r)\,\ell^s(r)\,\dd r} +\abs{J_1}+ \abs{J_2}\\
&\le  C_f(t-s)(1+\norm{\phi^s(\cdot\wedge t)}_\infty) + \abs{J_1}+\abs{J_2}\\
&\le \tilde{C}_0 (t-s)(1+\norm{x_0(\cdot\wedge s)}_\infty) +\abs{J_1}+\abs{J_2},
\end{align*}
where $\tilde{C}_0>0$ is a constant that is  independent of $s$, $t$, $x_0$,  $x$, and $a$, and
\begin{align*}
J_1&:=\int_t^T \chi^s(r)\,\Delta\ell^s(r)+\Delta\chi^s(r)\,\ell^t(r)\,\dd r,\\
J_2&:=\chi^s(T)\,\Delta h^s+\Delta \chi^s(T)\,h(\phi^t).
\end{align*}
To estimate $J_1$ and $J_2$, we will use the fact that, for each $r\in [t,T]$,
\begin{align*}
\Delta \chi^s(r)&=\left[e^{-\int_s^r \lambda^{s,x,a}(\theta)\,\dd \theta}
-e^{-\int_t^r \lambda^{s,x,a}(\theta)\,\dd \theta}\right]\\
&\qquad -\left[
e^{-\int_t^r \lambda^{t,x,a}(\theta)\,\dd \theta}
-e^{-\int_t^r \lambda^{s,x,a}(\theta)\,\dd \theta}
\right]\\
&=\left[
\int_s^t (-\lambda^{s,x,a})(\theta)\,e^{-\int_\theta^r \lambda^{s,x,a}(\tau)\,\dd \tau}\,\dd \theta
\right]\\&\qquad -\left[
e^{-\int_t^r \lambda^{t,x,a}(\theta)\,\dd \theta}
-e^{-\int_t^r \lambda^{t,(\phi^{s,x,a}),a}(\theta)\,\dd \theta}
\right]
\end{align*}
and thus, by Assumption~\ref{A:DOC}, by \eqref{E:chi:Davis}, and by Lemma~\ref{L:Appendix:Reg:t1},
\begin{equation}\label{E:chi:Claim2}
\begin{split}
\abs{\Delta \chi^s(r)}&\le C_\lambda (t-s)+ L_f Te^{L_fT} \norm{x(\cdot\wedge t)-\phi^{s,x,a}(\cdot\wedge t)}_\infty \\
&\le  C_\lambda(t-s)+ \tilde{C}_1(1+\norm{x_0(\cdot\wedge s_0)}_\infty) (t-s)
\end{split}
\end{equation}
for  some constant $\tilde{C}_1>0$ independent of $s$, $t$, $x_0$,  $x$, and $a$.
Now, we estimate $J_1$. By Assumption~\ref{A:DOC},  by Lemma~\ref{L:Appendix:Reg:t1},  by \eqref{E:chi:Claim2},
and by Lemma~\ref{L:Appendix:XL}, we have
\begin{align*}
\abs{J_1}&\le  TL_f \norm{\phi^s-\phi^t}_\infty + T(t-s)
[C_\lambda+ \tilde{C}_1(1+\norm{x_0(\cdot\wedge s_0)}_\infty)] C_f (1+\norm{\phi^{t,x,a}}_\infty)\\
&\le \tilde{C}_2(1+\norm{x_0(\cdot\wedge s)}_\infty)^2(t-s)
\end{align*}
for some constant $\tilde{C}_2>0$ independent of $s$, $t$, $x_0$,  $x$, and $a$. 
Finally, we estimate $J_2$. 
By Assumption~\ref{A:DOC}, by \eqref{E:chi:Claim2}, by Lemma~\ref{L:Appendix:Reg:t1},
and by Lemma~\ref{L:Appendix:XL}, we have, with $\mathbf{0}$ being the zero path in $\Omega$,
\begin{align*}
\abs{J_2}&\le L_f\norm{\phi^s-\phi^t}_\infty+ 
(t-s) [C_\lambda + \tilde{C}_1(1+\norm{x_0(\cdot\wedge s_0)}_\infty]\, (\abs{h(\mathbf{0})}+L_f\norm{\phi^t}_\infty)\\
&\le  \tilde{C}_3 (1+\norm{x_0(\cdot\wedge s)}_\infty)^2 (t-s)
\end{align*}
for some constant $\tilde{C}_3>0$ independent of $s$, $t$, $x_0$,  $x$, and $a$.
Consequently \eqref{E:Appendix:Claim2:2} holds,
which concludes the proof of Claim~2.\newline



Now, fix $L\ge 0$ and $x_0\in\Omega$. Let $0\le t_0\le t_1\le T$.
Let $x$, $\tilde{x}\in\mathcal{X}^L(t_0,x_0)$. Recall the constants $C_1$
and $C_2$ from Claims~1 and 2.
Put
\begin{align*}
C_{L,x_0}:=\max\{C_1[1+(LT+\norm{x_0(\cdot\wedge t_0)}_\infty) e^{LT}], C_2(1+\norm{x_0(\cdot\wedge t_0)})^2\}.
\end{align*}
By
 Claim~1 and  Lemma~\ref{L:Appendix:XL}, 
\begin{align*}
\abs{v(t_1,x)-v(t_1,\tilde{x})}\le  
\underbrace{C_1[1+(LT+\norm{x_0(\cdot\wedge t_0)}_\infty) e^{LT}]}_{
\le C_{L,x_0}
}\,\norm{x(\cdot\wedge t_1)-\tilde{x}(\cdot\wedge t_1)}_\infty.
\end{align*}
 Claim~2 states that
\begin{align*}
\abs{v(t_0,x)-v(t_1,x)}\le C_2(1+\norm{x_0(\cdot\wedge t_0)})^2(t_1-t_0)\le C_{L,x_0} (t_1-t_0).
\end{align*}

It remains to check continuity of $v$. To this end, fix $(s_1,x_1)\in\Omega$.
We will verify that $v$ is continuous at $(s_1,x_1)$.
Let $(s_2,x_2)\in [0,T]\times\Omega$.
We distinguish between two cases.

\textit{Case 1:} $s_1<s_2$. Then, noting that $x_1(\cdot\wedge s_1)\in\mathcal{X}^L(s_1,x_1)$ with $L=0$, we have
\begin{align*}
&\abs{v(s_1,x_1)-v(s_2,x_2)}\le  
\abs{v(s_1,x_1)-v(s_2,x_1(\cdot\wedge s_1))}+\abs{v(s_2,x_1(\cdot\wedge s_1))-v(s_2,x_2)}\\
&\qquad\qquad\le C_{0,x_1} (s_2-s_1)+C_1 (1+\norm{x_1(\cdot\wedge s_1)}_\infty)
\norm{x_1(\cdot\wedge s_1)-x_2(\cdot\wedge s_2)}_\infty
\end{align*}
thanks to Claims~1 and 2.

\textit{Case 2:} $s_2\le s_1$. Then, parallel to Case 1, we have
\begin{align*}
&\abs{v(s_1,x_1)-v(s_2,x_2)}\le 
\abs{v(s_2,x_2)-v(s_1,x_2(\cdot\wedge s_2))}+\abs{v(s_1,x_2(\cdot\wedge s_2))-v(s_1,x_1)}\\
&\qquad\qquad\le C_2(1+\norm{x_2(\cdot\wedge s_2)})^2 (s_1-s_2)\\
&\qquad\qquad\qquad\qquad + C_1 (1+\norm{x_1(\cdot\wedge s_1)}_\infty)
\norm{x_1(\cdot\wedge s_1)-x_2(\cdot\wedge s_2)}_\infty\\
&\qquad\qquad
\le C_2(1+\norm{x_1(\cdot\wedge s_1)}_\infty+\norm{x_1(\cdot\wedge s_1)-x_2(\cdot\wedge s_2)}_\infty)^2 (s_1-s_2)\\
&\qquad\qquad\qquad\qquad + C_1 (1+\norm{x_1(\cdot\wedge s_1)}_\infty)
\norm{x_1(\cdot\wedge s_1)-x_2(\cdot\wedge s_2)}_\infty
\end{align*}
thanks to Claims~1 and 2.

From the estimates in the two cases above, we can immediately deduce that $v$ is continuous at $(s_1,x_1)$.
This concludes the proof of Proposition~\ref{P:value:regular}.
\end{proof}
}

 \bibliographystyle{amsplain}
\bibliography{PDPI}

\textbf{\large Statements and Declarations.}

\textbf{Funding.}
The research of the first named author was partially supported by the 2018 GNAMPA-INdAM project \textit{Controllo ottimo stocastico con osservazione parziale: metodo di randomizzazione ed equazioni di Hamilton-Jacobi-Bellman sullo spazio di Wasserstein} and the 2024 GNAMPA-INdAM \textit{Problemi di controllo ottimo in dimensione infinita.}
The research of the second named author was
 supported in part  by NSF-grant DMS-2106077.
 
 \textbf{Competing Interests.} The authors declare that they do not have any competing interests.
 \end{document}